\documentclass[opre,nonblindrev]{informs3} % current default for manuscript submission

\OneAndAHalfSpacedXI % current default line spacing
\usepackage{algorithm}
\usepackage{algpseudocode}

\TheoremsNumberedThrough     % Preferred (Theorem 1, Lemma 1, Theorem 2)
\ECRepeatTheorems

\EquationsNumberedThrough    % Default: (1), (2), ...
\newcommand{\remove}[1]{}

\usepackage{multirow}
\usepackage{booktabs}
\usepackage{longtable}
\usepackage{array}
\usepackage{amsfonts}
\usepackage{algorithm}
\usepackage{color}
\usepackage{bm}

\renewcommand{\theARTICLETOP}{}
\begin{document}
%%%%%%%%%%%%%%%%

%\setcounter{page}{1}
%\DOI{10.1287/xxxx.0000.0000}%

\RUNAUTHOR{Dey et al.}

\RUNTITLE{Dual bounds for sparse PCA}

% \TITLE{Bundling Information Goods of Decreasing Value}
\TITLE{Using $\ell_1$-relaxation and integer programming to obtain dual bounds for sparse PCA }

\ARTICLEAUTHORS{%
\AUTHOR{Santanu S. Dey}
\AFF{School of Industrial and Systems Engineering, Georgia Institute of Technology, \EMAIL{santanu.dey@isye.gatech.edu}} %, \URL{}}
\AUTHOR{Rahul Mazumder}
\AFF{Operations Research Center, Massachusetts Institute of Technology, \EMAIL{rahulmaz@mit.edu}}
\AUTHOR{Guanyi Wang}
\AFF{School of Industrial and Systems Engineering, Georgia Institute of Technology, \EMAIL{gwang93@gatech.edu}}
} % end of the block

\ABSTRACT{%
Principal component analysis (PCA) is one of the most widely used dimensionality reduction tools in scientific data analysis.  The PCA direction, given by the leading eigenvector of a covariance matrix, is a linear combination of all features with nonzero loadings---this impedes interpretability. Sparse principal component analysis (SPCA) is a framework that enhances interpretability by incorporating an additional sparsity requirement in the feature weights (factor loadings) while finding a direction that explains the maximal variation in the data. However, unlike PCA, the optimization problem associated with the SPCA problem is NP-hard. Most conventional methods for solving SPCA are heuristics with no guarantees such as certificates of optimality on the solution-quality via associated dual bounds. Dual bounds are available via standard semidefinite programming (SDP) based relaxations, which may not be tight and the SDPs are difficult to scale using off-the-shelf solvers. In this paper, we present a convex integer programming (IP) framework to derive dual bounds. At the heart of our approach is the so-called $\ell_1$-relaxation of SPCA. While the $\ell_1$-relaxation leads to convex optimization problems for $\ell_0$-sparse linear regression and relatives; it results in a non-convex optimization problem for the PCA problem. We first show that the $\ell_1$-relaxation gives tight multiplicative bound on SPCA. Then we show how to use standard integer programming techniques to further relax the $\ell_1$-relaxation into a convex IP, for which there are good commercial solvers. We present worst-case results on the quality of the dual bound provided by the convex IP. We empirically observe that the dual bounds are significantly better than worst-case performance, and are superior to the SDP bounds on some real-life instances. Moreover, solving the convex IP model using commercial IP solvers appears to scale much better than solving the SDP-relaxation using commercial solvers. To the best of our knowledge, we obtain the best dual bounds for real and artificial instances for SPCA problems involving covariance matrices of size up to $2000\times 2000$.
}
%

%\KEYWORDS{deterministic inventory theory; infinite linear programming duality;
%  existence of optimal policies; semi-Markov decision process; cyclic schedule}

%\HISTORY{This paper was first submitted on April 12, 1922 and has been with the authors for 83 years for 65 revisions.}

\maketitle
%%%%%%%%%%%%%%%%%%%%%%%%%%%%%%%%%%%%%%%%%%%%%%%%%%%%%%%%%%%%%%%%%%%%%%

% Samples of sectioning (and labeling) in OPRE
% NOTE: (1) \section and \subsection do NOT end with a period
%       (2) \subsubsection and lower need end punctuation
%       (3) capitalization is as shown (title style).
%
%\section{Introduction.}\label{intro} %%1.
%\subsection{Duality and the Classical EOQ Problem.}\label{class-EOQ} %% 1.1.
%\subsection{Outline.}\label{outline1} %% 1.2.
%\subsubsection{Cyclic Schedules for the General Deterministic SMDP.}
%  \label{cyclic-schedules} %% 1.2.1
%\section{Problem Description.}\label{problemdescription} %% 2.

% Text of your paper here

\section{Introduction}
Principal component analysis (PCA) is one of the most widely used dimensionality reduction methods in data science. Given a data matrix $Y \in \mathbb{R}^{m \times n}$ (with $m$ samples and $n$ features; and each feature is centered to have zero mean), PCA seeks to find a principal component (PC) direction $x \in \mathbb{R}^n$ with $\|x\|_2 = 1$ that maximizes the variance of a weighted combination of features. Formally, this PC direction can be found by solving
\begin{align}
	\max_{\|x\|_2 = 1} x^{\top} A x \tag{PCA} \label{eq:PCA}
\end{align}
where $A \triangleq \frac{1}{m} Y^{\top} Y$ is the sample covariance matrix.
An obvious drawback of PCA is that all the entries of $\hat{{x}}$ (an optimal solution to~\eqref{eq:PCA}) are (usually) nonzero, which leads to the PC direction being a linear combination of all features -- this impedes interpretability~\cite{cadima1995loading,jolliffe2003modified,zou2006sparse}. In biomedical applications for example, when ${Y}$ corresponds to the gene-expression measurements for different samples, it is desirable to obtain a PC direction which involves only a handful of the features (e.g, genes) for interpretation purposes. In financial applications (e.g, $A$ may denote a covariance matrix of stock-returns), a sparse subset of stocks that are responsible for driving the first PC direction may be desirable for interpretation purposes.  Indeed, in many scientific and industrial applications~\cite{allen2011sparse,hastie2015statistical,wht_09},
%~\cite{wang2013online,naikal2011informative,majumdar2009image,chan2010using,allen2011sparse}, 
for additional interpretability, it is desirable for the factor loadings to be sparse, i.e., few of the entries in $\hat{x}$ are nonzero and the rest are zero. This motivates the notion of a sparse principal component analysis (SPCA)~\cite{hastie2015statistical,jolliffe2003modified}, wherein, in addition to maximizing the variance, one also desires the direction of the first PC to be sparse in the factor loadings. The most natural optimization formulation of this problem, modifies criterion~\ref{eq:PCA} with an additional sparsity constraint on $x$ leading to:
\begin{align}
	\lambda^k(A) \triangleq \max_{\|x\|_2 = 1, \|x\|_0 \leq k} x^{\top} A x \tag{SPCA} \label{eq:SPCA}
\end{align}
where $\|x\|_0 \leq k$, is equivalent to allowing at most $k$ components of $x$ to be nonzero. Unlike the \ref{eq:PCA} problem, the \ref{eq:SPCA} problem is NP-hard~\cite{chan2015worst, magdon2017np}.

Many heuristic algorithms have been proposed in the literature that use greedy methods~\cite{jolliffe2003modified,zhang2002low,he2011algorithm,d2008optimal},  alternating methods~\cite{yuan2013truncated} and the related power methods~\cite{journee2010generalized}. However, conditions under which (some of) these computationally friendlier methods can be shown to work well, make very strong and often unverifiable assumptions on the problem data. Therefore, the performance of these heuristics (in terms of how close they are to an optimal solution of the ~\ref{eq:SPCA} problem) on a given dataset is not clear. 

Since \ref{eq:SPCA} is NP-hard, there has been exciting work in the statistics community~\cite{berthet2013optimal,wang2016statistical} in understanding the statistical properties of convex relaxations (e.g., those proposed by~\cite{jordan_07} and variants)  of \ref{eq:SPCA}. It has been established~\cite{berthet2013optimal,wang2016statistical} that the statistical performance of estimators available from convex relaxations are sub-optimal (under suitable modeling assumptions) when compared to estimators obtained by (optimally) solving~\ref{eq:SPCA}---this further underlines the importance of creating tools to be able to solve \ref{eq:SPCA} to optimality. 

Our main goal in this paper is to propose an integer programming framework that allows the computation of certificates of optimality via dual bounds, which make limited restrictive/unverifiable assumptions on the data. Dual bounds can also translate into suitable guarantees for statistical performance of the estimator---see for example,~\cite{mazumder2017discrete}[Theorem 4] for results pertaining to approximate solutions for sparse regression settings\footnote{In~\cite{mazumder2017discrete}, estimators with certificates on dual bounds translate to simple modifications of error bounds that correspond to the global solution of the original nonconvex estimator.}.
To the best of our knowledge, the only published methods for obtaining dual bounds of \ref{eq:SPCA} are based on semidefinite programming (SDP) relaxations~\cite{d2005direct,d2014approximation,d2008optimal,zhang2012sparse} (see Appendix \ref{sec:appsdp} for the SDP relaxation) and spectral methods involving a low-rank approximation of the matrix $A$~\cite{papailiopoulos2013sparse}. Both these approaches however, have some limitations. 
The SDP relaxation does not appear to scale easily (using off-the-shelf solver {\color{black} Mosek 8.0.0.60}) for matrices with more than a few hundred rows/columns, while applications can be significantly larger.  Indeed, even a relatively recent implementation based on the Alternating Direction Method of Multipliers for solving the SDP considers instances with  $n\approx 200$~\cite{Ma2013}.  
The spectral methods involving a low-rank approximation of $A$ proposed in~\cite{papailiopoulos2013sparse} have a running time of $\mathcal{O}(n^d)$ where $d$ is the rank of the matrix---in order to scale to large instances, no more than a rank $2$ approximation of the original matrix seems possible. The paper \cite{bertsimas_berk_2016} presents a specialized branch and bound solver\footnote{This paper is not available in the public domain at the time of writing this paper.} to obtain solutions to the SPCA problem, but their method can handle problems with $n\approx 100$ -- the approach presented here is different, and our proposal scales to problem instances that are much larger. 
	
	The methods proposed here are able to obtain approximate dual bounds of \ref{eq:SPCA} by solving convex integer programs and a related perturbed version of convex integer programs that are easier to solve. The dual bounds we obtain are incomparable to dual bounds based on the SDP relaxation, i.e. neither dominates the other, and the method appears to scale well to matrices up to sizes of $2000 \times 2000$.

\section{Main results}
In this paper, we use upper case letters such as $A, X$ to denote symmetric matrices. The $(i,j)$-th component of matrix $A$ is denoted as $[A]_{ij}$ or $A_{ij}$ in short. We use lower case letters such as $v, x$ for vectors, and denote the $i$-th component of a vector $v$ as $[v]_i$ or $v_i$ in short. We use upper case letter $I$ for set of indices. Given a vector  where $v \in \mathbb{R}^n$ and $I \subseteq [n]$, we let $v_I \in \mathbb{R}^n$ to be the vector: 
\begin{align*}
	[v_I]_i = \left\{
	\begin{array}{lll}
		v_i & i \in I \\
		0 & i \notin I \\
	\end{array}
	\right.
\end{align*}
We use the usual notation $\|\cdot\|_1, ~ \|\cdot\|_2$ for $\ell_1, ~ \ell_2$ norm respectively for a given vector. Let $\|\cdot\|_0$ be the $\ell_0$ norm which denotes the number of non-zero components. Given a set $S$, we denote $\textup{conv}(S)$ as the convex hull of $S$; given a positive integer $n$ we denote $\{1, \dots, n\}$ by $[n]$; given a matrix $A$, we denote its trace by $\textup{tr}(A)$. {\color{black} Given $n$ scalars $v_1, \dots v_n$, $\text{diag}(v_1, \dots, v_n)$ is the $n \times n$ matrix whose diagonal elements are $v_i$'s and the off-diagonal terms are equal to $0$. } We list all the notation used in this paper in Table~\ref{tab:notation}.

Notice that the constraint $\|x\|_2 = 1, \|x\|_0 \leq k$ implies that $\|x\|_1 \leq \sqrt{k}$. Thus, one obtains the so-called $\ell_1$-norm relaxation of \ref{eq:SPCA}:
\begin{align}
\text{OPT}_{\ell_1} \triangleq \max_{\|x\|_2 \leq 1, \|x\|_1 \leq \sqrt{k}} x^{\top} A x \tag{$\ell_1$-relax} \label{eq:ell_1-relax}.
\end{align} 
The relaxation \ref{eq:ell_1-relax} has two advantages: 
\begin{itemize}
\item[(a)] As shown in Theorem~\ref{thm:app-ratio} below, \ref{eq:ell_1-relax} gives a constant factor bound on \ref{eq:SPCA}, 
\item[(b)] The feasible region is convex and all the nonconvexity is in the objective function. 
\end{itemize}
\textcolor{black}{We build on these two advantages: our convex IP relaxation is a further relaxation of \ref{eq:ell_1-relax} (together with some implied linear inequalities for \ref{eq:SPCA}) which heavily use the fact that the feasible region of \ref{eq:ell_1-relax} is convex.} {\color{black}We require to use IP methods and construct the convex IP, since the objective of \ref{eq:ell_1-relax} is non-convex. Thus, we use a combination of \ref{eq:ell_1-relax} and IP methods to obtain strong dual bounds.}

We note that \ref{eq:ell_1-relax} is an important estimator in its own right~\cite{hastie2015statistical,wht_09}---it is commonly used in the statistics/machine-learning community as one that leads to an eigenvector of $A$ with entries having a small $\ell_{1}$-norm (as opposed to a small $\ell_0$-norm). {\color{black} We emphasize that $\ell_{1}$-relaxation has never been used to computationally obtain dual bounds for SPCA. Indeed, to the best of our knowledge there has been no systematic study of the theoretical and empirical computational properties of the $\ell_1$-relaxation vis-\`a-vis SPCA.} 

The rest of this section is organized as follows:  In Section~\ref{sec:l1}, we present the constant factor bound on \ref{eq:SPCA} given by \ref{eq:ell_1-relax}, improving upon some known results. In Section~\ref{sec:cip}, we present the construction of our convex IP and prove results on the quality of bound provided. In Section~\ref{sec:cip2}, we discuss perturbing the original matrix in order to make the convex IP more efficiently solvable while still providing reasonable dual bounds.
In Section~\ref{sec:num-exp}, we present results from our computational experiments. 

\subsection{Quality of $\ell_1$-relaxation as a surrogate for the \ref{eq:SPCA} problem}\label{sec:l1}
The following theorem is an improved version of a result appearing in~\cite{vershynin2016high} (Exercise 10.3.7).
\begin{theorem} \label{thm:app-ratio}
	The objective value $\text{OPT}_{\ell_1}$ is upper bounded by a multiplicative factor $\rho^2$ away from $\lambda^k(A)$, i.e., $\lambda^k(A) \leq \text{OPT}_{\ell_1} \leq \rho^2 \cdot \lambda^k(A)$ with $\rho \leq 1 + \sqrt{\frac{k}{k + 1}} $. 
\end{theorem}
Proof of Theorem~\ref{thm:app-ratio} is provided in Section~\ref{sec:app1}. While we have improved upon the bound presented in~\cite{vershynin2016high}, we do not know if this new bound is tight. 

The approximation ratio $1 + \sqrt{\frac{k}{k + 1}}$ from Theorem 1 yields an almost $100 \%$ gap (see formal definition of gap in Section 4) in the worst case. From a practitioners' viewpoint, a $100 \%$ gap is obviously far from ideal and would not be considered as ``solving'' the problem. However, as we shall see in Section 4, the $\ell_1$-relaxation does provide very good dual bounds in many instances. Moreover, as stated above the approximation ratio of $1 + \sqrt{\frac{k}{k + 1}}$ is the best we can prove; however this bound may be significantly away from the actual bound.

Theorem \ref{thm:app-ratio} has implications regarding existence of polynomial-time algorithms to obtain a constant-factor approximation guarantee for \ref{eq:ell_1-relax}. In particular, the proof of Theorem \ref{thm:app-ratio} implies that if one can obtain a solution for \ref{eq:ell_1-relax} which is within a constant factor, say $\theta$, of $\text{OPT}_{\ell_1}$, then a solution for \ref{eq:SPCA} problem can be obtained, which is within a constant factor (at most $\theta\rho \approx 4\theta$) of $\lambda^k(A)$. Therefore, the $\ell_1$-relaxation is also {\color{black} inapproximable} in general.

\subsection{{\color{black}From $\ell_1$-relaxation to convex integer programming model}}\label{sec:cip}
A classical integer programming approach to finding dual bounds of \ref{eq:SPCA} would be to go to an extended space involving the product of $x$-variables and include one binary variable per $x$-variable in order to model the $\ell_0$-norm constraint, resulting in a very large number of binary variables. In particular, a typical model could be of the form:
\begin{eqnarray}
\textup{max} &  \textup{tr}(AX) \\
\textup{s.t.} & -z_i \leq x_i \leq z_i, ~ i \in [n] \\
& \sum\limits_{j = 1}^n z_i \leq k \\
& \|x\|_2 \leq 1 \\
& \left[\begin{array}{cc} 1 & x^{\top} \\ x & X\end{array}\right] \succeq 0\\
& \textup{rank}\left(\left[\begin{array}{cc} 1 & x^{\top} \\ x & X\end{array}\right]\right) = 1\\
& z \in \{0, 1\}^n.
\end{eqnarray}
It is easy to see that such a model is challenging due to {\color{black} (a) $n$ binary variables (b) ``quadratic" increase in number of variables ($X$) and (c) the presence of the rank constraint.} Even with significant progress, it is well-known that solving such problems beyond $n$ being a few hundred variables is extremely challenging~\cite{bienstock1996computational,frangioni2007sdp}. Indeed, instances with an arbitrary quadratic objective and bound constraints cannot be generally solved (exactly) by modern state-of-the-art methods as soon as the number of variables exceed a hundred or so~\cite{burer2009globally,bonami2016solving}.

{\color{black} This is how we address the challenges discussed above. 
\begin{enumerate}
\item $n$ binary variables (a): the feasible region of \ref{eq:ell_1-relax} is a convex set. Therefore, we do not have to include binary variables to model the $\ell_0$-norm constraint. We will use \ref{eq:ell_1-relax} as our basic relaxation. 
\item Quadratic increase in number of variables (b) and rank constraint (c): We do not use the $X$ variables to model the quadratic objective. Instead we upper bound the quadratic objective using piecewise linear function via integer programming techniques. 
\end{enumerate}

In other words, since the feasible region of \ref{eq:ell_1-relax} is a convex set and takes care of challenge (a), we model/upper bound the objective function using IP techniques to deal with challenges (b) and (c). Specifically, we follow the following procedure:}
\begin{itemize}
	\item[\textbf{step-0}:] By spectral decomposition, let $A = \sum_{i = 1}^n \lambda_i v_i v_i^{\top}$ where $(\lambda_i)_{i = 1}^n, (v_i)_{i = 1}^n$ are unit norm orthogonal eigen-pairs. Then the objective function of \ref{eq:ell_1-relax} is: $$\sum_{i = 1}^n \lambda_i (x^{\top} v_i)^2.$$ 
	\item[\textbf{step-1}:] Assuming that $\lambda \leq \lambda^k(A)$, we have that $x^{\top}Ax = x^{\top}(A  - \lambda I)x + \lambda$ for $x$ such that $\|x\|_2 = 1$, where $I $ is the identity matrix. Therefore, if we split the eigenvalues into two sets as $\{i: \lambda_i > \lambda \}$ and $\{i : \lambda_i < \lambda \}$,  the objective function can be represented as 
		\begin{align*}
			\lambda + \sum_{i \in \{i: \lambda_i > \lambda\}} (\lambda_i - \lambda) (x^{\top} v_i)^2 + \sum_{i \in \{i: \lambda_i < \lambda\}} (\lambda_i - \lambda) (x^{\top} v_i)^2
		\end{align*}
		where for each eigenvalue $\lambda_i$ that equals to $\lambda$, since $\lambda_i - \lambda = 0$, it does not contribute anything to objective function. %, one can ignore $(\lambda_i - \lambda)(x^{\top} v_i)^2 = 0 (x^{\top} v_i)^2$ in the summation. 
{\color{black} Note that the first term is convex and the second term is concave. Since the objective is a maximizing, we need to deal with the first term. This idea of splitting the objective function into convex and concave part is a well-studied approach for attacking non-convex quadratic objective functions.} See for example~\cite{bomze2013copositivity, burer2009old} for use of some similar ideas.
	\item[\textbf{step-2}:] \label{sec:cip-point3} For each index $i \in \{i: \lambda_i > \lambda \}$, replace $x^{\top} v_i$ with a single continuous variable $g_i$, and set $\theta_i \gets \max \{ x^{\top} v_i: \|x\|_2 \leq 1, \|x\|_0 \leq k \}$ (or $\theta_i \gets \max \{ x^{\top} v_i: \|x\|_2 \leq 1, \|x\|_1 \leq \sqrt{k} \}$ if we explicitly want a relaxation of \ref{eq:ell_1-relax}) be an upper bound of $g_i$.  Then for each $g_i$ with $i \in \{i: \lambda_i > \lambda \}$, construct a piecewise linear upper approximation $\xi_i$ for $g_i^2$. Such piecewise linear upper approximation is usually modelled via \textit{special ordered sets of type 2} (SOS-2) constraints~\cite{nemhauser1988integer}. 
	\item[\textbf{step-3}:] For $\sum_{i \in \{i: \lambda_i < \lambda\}} ( \lambda_i - \lambda) (x^{\top} v_i)^2$, since $\lambda_i - \lambda < 0$, we obtain a convex constraint $\sum_{i \in \{i: \lambda_i < \lambda \}} - (\lambda_i - \lambda) (x^{\top} v_i)^2 \leq s$.    
\end{itemize}
{\color{black}
Therefore, a convex integer programming problem is obtained as follows:
\begin{align*}
	\begin{array}{rllll}
		\text{OPT}_{\text{convex-IP}} \triangleq \max & ~ \lambda + \sum_{i \in \{i : \lambda_i > \lambda\}} (\lambda_i - \lambda) \xi_i - s & \\
		\text{s.t. } & \left\{
		\begin{array}{lll}
		g_i = x^{\top} v_i \\
		- \theta_i \leq g_i \leq \theta_i
		\end{array}
		\right. & i \in [n] \\
		& \left\{
		\begin{array}{lll}
		g_i = \sum_{j = - N}^N \gamma_i^j \eta_i^j \\
		\xi_i = \sum_{j = - N}^N (\gamma_i^j)^2\eta_i^j \\
		(\eta_i^{-N}, \ldots, \eta_i^N) \in \text{SOS-2} 
		\end{array}
		\right. & i \in \{i: \lambda_i > \lambda \} \\
		& \left\{
		\begin{array}{lll}
		\sum_{i = 1}^n x_i^2 \leq 1 \\
		\sum_{i \in \{i : \lambda_i > \lambda\}} \left( \xi_i - \frac{\theta_i^2}{4N^2} \right)+ \sum_{i \in \{i : \lambda_i \leq \lambda\}} g_i^2 \leq 1
		\end{array}
		\right. & \\
		& \left\{
		\begin{array}{lll}
		\sum_{i = 1}^n y_i \leq \sqrt{k} \\
		y_i \geq x_i, ~ y_i \geq - x_i, ~ \forall i \in [n]
		\end{array}
		\right. & \\
		& ~ \sum_{i \in \{i: \lambda_i < \lambda \}} - (\lambda_i - \lambda) g_i^2 \leq s
	\end{array} \tag{Convex-IP} \label{eq:convex-IP}
\end{align*}

\textbf{Notations and explanations of \ref{eq:convex-IP}:}
\begin{description}
	\item[Variable $g_i$:] The first set of constraints 
		\begin{align*}
			\left\{
			\begin{array}{lll}
			g_i = x^{\top} v_i \\
			- \theta_i \leq g_i \leq \theta_i
			\end{array}
			\right. 
		\end{align*}
		transfers $x^{\top} v_i$ into a single variable for each $i \in [n]$. 
	\item[Variable $\xi_i$:] Based on step-2 above, for each $i \in \{i: \lambda_i > \lambda \}$, the second set of constraints 
		\begin{align*}
			\left\{
			\begin{array}{lll}
			g_i = \sum_{j = - N}^N \gamma_i^j \eta_i^j \\
			\xi_i = \sum_{j = - N}^N (\gamma_i^j)^2\eta_i^j \\
			(\eta_i^{-N}, \ldots, \eta_i^N) \in \text{SOS-2} 
			\end{array}
			\right.
		\end{align*}
		forms $\xi_i$ as a piecewise-linear upper approximation of $g_i^2$. Let $2N + 1$ be the number of splitting points of the domain $[- \theta_i, \theta_i]$ of variable $g_i$, where the set of splitting points $(\gamma_i^j)_{j = - N}^N$ satisfy  
		\begin{align*}
			- \theta_i = \gamma_i^{-N} < \ldots \gamma_i^0 ~ ( = 0) < \ldots < \gamma_i^N = \theta_i.
		\end{align*}
		Without any prior information of the optimal solution, we partition the set $[- \theta_i, \theta_i]$ equally to minimize the (worst-case) upper bounds, i.e.,  by letting $(\gamma_i^j)_{j = - N}^N \gets \left( \frac{j}{N} \cdot \theta_i \right)_{j = - N}^N$ be the value of $j^{\text{th}}$ splitting point. See Section~\ref{sec:app3} for details.
	\item[Quadratic constraints:] The third set of constraints does the following: Since $v_i$'s are orthogonal, then $\sum_{i = 1}^n x_i^2 \leq 1$ implies $\sum_{i = 1}^n g_i^2 \leq 1$. Together with $\xi_i$ representing $g_i^2$, we can obtain the implied inequality: 
		\begin{align*}
			\sum_{i \in  \{ i : \lambda_i > \lambda \}} \xi_i + \sum_{i \in \{i: \lambda_i \leq \lambda \}} g_i^2 \leq 1 + \sum_{i \in \{i: \lambda_i > \lambda\}} \frac{\theta_i^2}{4 N^2}
		\end{align*}
		The second term in the right-hand-side reflects the fact that $\xi_i$ is not exactly equal to  $g_i^2$, but only a piecewise linear upper bound of $g_i^2$. Note that the exact value of the second term in the right-hand-side also depends on the way one splits the set $[- \theta_i, \theta_i]$, the value $\sum_{i \in \{i: \lambda_i > \lambda\}} \frac{\theta_i^2}{4 N^2}$ in above formula is obtained via splitting $[- \theta_i, \theta_i]$ equally, which can be shown as the minimum upper bounds without any prior idea of the optimal solution $x$ of \ref{eq:SPCA} or \ref{eq:ell_1-relax}. See the proof in Section~\ref{sec:app3} for details. %, in fact, we have $x_i \leq g_i^2 + \frac{\theta_i^2}{4 N^2}$ for $i \in \{i: \lambda_i > \lambda\}$. 
This constraint (cutting-plane) is not necessarily needed for a correct model -- it is used since it helps improving the dual bound of the LP relaxation and significantly improves the running-time of the solver. 
	\item[$\ell_1$ constraints:] The fourth set of constraints (the fourth one within the curly brackets in Convex IP) introduce new variables $y_i$ to denote $|x_i|$ for $i = 1, \ldots, n$ and model the constraint $$ \sum_{i = 1}^n |x_i| \leq \sqrt{k}.$$ 
	\item[Convex constraint:] The final constraint 
		\begin{align*}
			\sum_{i \in \{i: \lambda_i < \lambda\}} - (\lambda_i - \lambda) g_i^2 \leq s \tag{convex-constraint} \label{eq:convex-constraint}
		\end{align*}
		is a convex constraint that we obtained in step-3 where $x^{\top} v_i$ is replaced by a variable $g_i$ since $g_i = x^{\top} v_i$.
\end{description}
} 
We arrive at the following result:
\begin{proposition} \label{prop:SPCA-bound}
	The optimal objective value $\text{OPT}_{\text{convex-IP}}$ of \ref{eq:convex-IP} is an upper bound on the \ref{eq:SPCA} problem.
\end{proposition}
Proposition~\ref{prop:SPCA-bound} is formally verified in Appendix~\ref{sec:app2}. 

Next combining the result of Theorem~\ref{thm:app-ratio} with the quality of the approximation of the objective function of \ref{eq:ell_1-relax} by \ref{eq:convex-IP}, we obtain the following result:

\begin{proposition} \label{prop:convex-IP-bound}
	The optimal objective value $\text{OPT}_{\text{convex-IP}}$ of \ref{eq:convex-IP} is upper bounded by 
	\begin{align*}
		\text{OPT}_{\textup{convex-IP}} \leq \rho^2 \lambda^k(A) + \frac{1}{4N^2} \sum_{i \in \{i: \lambda_i > \lambda\}} (\lambda_i - \lambda) \theta_i^2.
	\end{align*}
\end{proposition}
A proof of Proposition~\ref{prop:convex-IP-bound} is presented in Appendix~\ref{sec:app3}.

Finally, let us discuss why we expect \ref{eq:convex-IP} to be appealing from a computational viewpoint. Unlike typical integer programming approaches, the number of binary variables in \ref{eq:convex-IP} is $(2N + 1) \cdot |\{i: \lambda_i > \lambda \}|$ which is usually significantly smaller than $n$. Indeed, heuristics for SPCA generally produce good values of $\lambda$, and in  almost all experiments we found that $|\{i: \lambda_i > \lambda \}|  \ll n$. Moreover, $N$ is a parameter we control. In order to highlight the ``computational tractability" of \ref{eq:convex-IP}, we formally state the following result: 

\begin{proposition} \label{prop:convex-IP-poly}
	Assuming the number of splitting points $N$ and the size of set $\{i: \lambda_i > \lambda\}$ is fixed, the \ref{eq:convex-IP} problem can be solved in polynomial time. 
\end{proposition}

\textcolor{black}{Note that the convex integer programming method which is solvable in polynomial time, does not contradict the inapproxamability of the \ref{eq:SPCA} problem, since $\text{OPT}_{\text{convex-IP}}$ is upper bounded by the sum of $\rho^2 \lambda^k(A)$ and a term corresponding to the sample covariance matrix.} 

\subsection{Improving the running time of Convex-IP}\label{sec:cip2}
\subsubsection{Perturbation of the covariance matrix $A$:}
%%%In practice, 
{\color{black} In practice, we do the following (sequence of) perturbation on covariance matrix $A$ to reduce the running time of solving convex IP. Again let $\lambda$ (obtained from some heuristic method) be a lower bound on the $\lambda^k(A)$, let $A = \sum_{i = 1}^n \lambda_i v_i v_i^{\top}$ be the spectral decomposition of $A$ with $\lambda_1 \geq \ldots \geq \lambda_n \geq 0$. }
\begin{enumerate}
	\item Set $\bar{\lambda} \triangleq \max \{\lambda_i: \lambda_i \leq \lambda \}$ (where $\lambda_1, \lambda_2, \dots, \lambda_n$ are the eigenvalues of $A$). We assume $\bar{\lambda} < \lambda$. {\color{black}However, when $\bar{\lambda} \triangleq \max\{\lambda_i: \lambda_i \leq \lambda \} = \lambda$, one can apply Algorithm~\ref{algo:pert} to obtain a matrix $\bar{A} \succeq A$ such that none of the eigenvalues of $\bar{A}$ equals $\lambda$. We then replace $A$ by $\bar{A}$. Now letting $\lambda_1, \lambda_2, \dots, \lambda_n$ to be  the eigenvalues of (the updated) $A$ and $\bar{\lambda} \triangleq \max \{\lambda_i: \lambda_i \leq \lambda \}$, we obtain that $\bar{\lambda} < \lambda$ for $\bar{A}$. }

\begin{algorithm}
\caption{\textcolor{black}{Perturbation of $A$}} \label{algo:pert}
\begin{algorithmic}[1]
\State \emph{Input}: Sample covariance matrix $A$ and $\lambda$. 
\State \emph{Output}: A perturbed sample covariance matrix $\bar{A}$ with distinct eigenvalues such that $\bar{A} \succeq A$ {\color{black} and none of the eigenvalues of $\bar{A}$ equals $\lambda$}.
\Function{Perturbation Method}{$A, \lambda$}\label{function:heuristic}
\State Compute spectral decomposition on $A$ as $A = V^{\top} \Lambda V$, {\color{black} where $\Lambda = \text{diag}(\lambda_1, \dots, \lambda_n)$. Let $\lambda_{i_1} > \cdots > \lambda = \lambda_{i_j} > \cdots \lambda_{i_p} \geq 0$ be all its distinct values of eigenvalues where $p \leq n$.}
\State Set $\Delta \lambda \gets \min \{\lambda_{i_j} - \lambda_{i_{j + 1}}\,|\, j = 1, \ldots, p - 1\}$.
\State Set $\bar{\Lambda} \gets \Lambda + \text{diag}\left( \frac{i - 1}{n} \epsilon \,|\, i = n, \ldots, 1 \right)$ {\color{black} with $\epsilon = \frac{1}{2}\Delta \lambda$.}
\State \Return $\bar{A} \gets V^{\top} \bar{\Lambda} V$.
\EndFunction
\end{algorithmic}
\end{algorithm}		
	\item Perturb the covariance matrix $A = \sum_{i = 1}^n \lambda_i v_i v_i^{\top}$ by $\bar{A} = \sum_{i \in \{i: \lambda_i > \lambda\}} \lambda_i v_i v_i^{\top} + \sum_{i \in \{i: \lambda_i \leq \lambda\}} \bar{\lambda} v_i v_i^{\top}$. Note that the objective value $\text{OPT}_{\text{convex-IP}}(\bar{A})$ in \ref{eq:convex-IP} is an upper bound on $\text{OPT}_{\text{convex-IP}}(A)$. {\color{black} This is because if $(x, y, g, \xi, \eta, s)$ is a feasible solution of \ref{eq:convex-IP}, then the objective function value of \ref{eq:convex-IP} corresponding to $\bar{A}$ is at least as large as that of $A$.} Replace $A$ by $\bar{A}$. 
	\item Therefore, the convex constraint $\sum_{i \in \{i: \lambda_i \leq \lambda\}} - (\lambda_i - \lambda) g_i^2 \leq s$ in \ref{eq:convex-IP} can be replaced by $\sum_{i \in \{i: \lambda_i \leq \lambda\}} - (\bar{\lambda} - \lambda) g_i^2 \leq s$, i.e., $\sum_{i \in \{i: \lambda_i \leq \lambda\}} g_i^2 \leq \dfrac{s}{\lambda - \bar{\lambda}}$. 
	\item {\color{black} Let $(\bar{x}, \bar{y}, \bar{g}, \bar{\xi}, \bar{\eta}, \bar{s})$ be an optimal solution for \ref{eq:convex-IP}. Since the convex constraint achieves equality for any optimal solution of \ref{eq:convex-IP}, i.e., 
		\begin{align*}
			& \sum_{i \in \{i: \lambda_i \leq {\lambda}\}} - ({\lambda} - \bar{\lambda}) \bar{g}_i^2 = \bar{s}
		\end{align*}
		together with	
		\begin{align*}
		&\sum_{i = 1}^n \bar{g}_i^2 = \sum_{i \in \{i: \lambda_i \leq {\lambda}\}} \bar{g}_i^2 +  \sum_{i \in \{i: \lambda_i > {\lambda}\}} \bar{g}_i^2 \leq 1 \\
			& 1\leq \sum_{i \in \{i: \lambda_i >  {\lambda} \}} \bar{\xi}_i + \sum_{i \in \{i: \lambda_i \leq  {\lambda} \}} \bar{g}_i^2 \leq 1 + \frac{1}{4N^2} \sum_{i \in \{i: \lambda_i > {\lambda} \}} \theta_i^2,
		\end{align*}
		imply the following inequalities:
		\begin{align*}
			& 1 - \dfrac{\bar{s}}{{\lambda} - \bar{\lambda}} \leq \sum_{i \in \{i: \lambda_i > {\lambda} \}} \bar{\xi}_i \leq 1 + \frac{1}{4N^2} \sum_{i \in \{i: \lambda_i > {\lambda}\}} \theta_i^2 - \dfrac{\bar{s}}{{\lambda} - \bar{\lambda}}, \\
			& \sum_{i \in \{i: \lambda_i > {\lambda}\}} \bar{g}_i^2 \leq 1 - \dfrac{\bar{s}}{{\lambda} - \bar{\lambda}}.
		\end{align*}
}

\end{enumerate}
Thus a simplified convex IP corresponding to the perturbed {\color{black} covariance matrix is:
\begin{align*}
	\begin{array}{rllll}
		\text{OPT}_{\text{pert-convex-IP}} \triangleq \max & ~ {\lambda} + \sum_{i \in \{i : \lambda_i > {\lambda}\}} (\lambda_i - {\lambda}) \xi_i - s & \\
		\text{s.t. } & \left\{
		\begin{array}{lll}
		g_i = x^{\top} v_i \\
		- \theta_i \leq g_i \leq \theta_i
		\end{array}
		\right. & i \in \{i: \lambda_i > {\lambda} \} \\
		& \left\{
		\begin{array}{lll}
		g_i = \sum_{j = - N}^N \gamma_i^j \eta_i^j \\
		\xi_i = \sum_{j = - N}^N (\gamma_i^j)^2\eta_i^j \\
		(\eta_i^{-N}, \ldots, \eta_i^N) \in \text{SOS-2} 
		\end{array}
		\right. & i \in \{i: \lambda_i > {\lambda} \} \\
		& \left\{
		\begin{array}{lll}
		\sum_{i = 1}^n x_i^2 \leq 1 \\
		\sum_{i \in \{i: \lambda_i > {\lambda}\}} g_i^2 \leq 1 - \frac{s}{{\lambda} - \bar{\lambda}} \\
		1 - \frac{s}{{\lambda} - \bar{\lambda}} \leq \sum_{i \in \{i: \lambda_i > {\lambda}\}} \xi_i \leq 1 + \sum_{i \in \{i: \lambda_i > {\lambda}\}}\frac{\theta_i^2}{4 N^2} - \frac{s}{{\lambda} - \bar{\lambda}}
		\end{array}
		\right. & \\
		& \left\{
		\begin{array}{lll}
		\sum_{i = 1}^n y_i \leq \sqrt{k} \\
		y_i \geq x_i, ~ y_i \geq - x_i, ~ \forall i \in [n]
		\end{array}
		\right. & \\
		& ~ v^{\top} y \leq b_{(v)}
	\end{array} \tag{Pert-Convex-IP} \label{eq:pert-convex-IP}
\end{align*}
where the quadratic constraints in \ref{eq:pert-convex-IP} are updated based on the discussion above and the final constraint $v^{\top} y \leq b_{(v)}$ represents the cutting planes that we add, see Proposition~\ref{prop:cut} for details. }

\begin{proposition} \label{prop:Pert-Convex-IP}
	The optimal objective value $\text{OPT}_{\text{Pert-Convex-IP}}$  is upper bounded by 
	\begin{align*}
		& ~ \text{OPT}_{\textup{Pert-Convex-IP}} \leq \rho^2 \lambda^k(A) +  \rho^2 (\bar{\lambda} - \lambda_{\min} (A))+ \frac{1}{4N^2} \sum_{i \in \{i: \lambda_i > \lambda\}} (\lambda_i - \lambda) \theta_i^2.
 	\end{align*}
\end{proposition}

Note that in \ref{eq:pert-convex-IP}, we do not need the variables $g_i, i \in \{i: \lambda_i \leq \lambda\}$ which greatly reduces the number of variables since in general $|\{i: \lambda_i \geq \lambda\}| \ll n$. 
In practice, we note a significant reduction in running time, while the dual bound obtained from \ref{eq:pert-convex-IP} model remains reasonable. More details are presented in Section \ref{sec:num-exp}.

\subsubsection{Refining the splitting points}
Since the \ref{eq:pert-convex-IP} model runs much faster than the \ref{eq:convex-IP} model, we run the \ref{eq:pert-convex-IP} model iteratively. In each new iteration, we  add one extra splitting point describing each $\xi_i$ function. In particular, once we solve the \ref{eq:pert-convex-IP} model, we add one splitting point at the optimal value of $g_i$.

\subsubsection{Cutting planes}
\begin{proposition}\label{prop:cut}
Let $x \in \mathbb{R}^n$. Let $|x_{i_1}| \geq |x_{i_2}| \geq \dots \geq |x_{i_{n-1}}| \geq |x_{i_n}|$. Then let ${v}$ be the vector:
\begin{eqnarray}
{v}_{i_j} = \left\{ \begin{array}{rl} |x_{i_j}| & \textup{ if } j \leq k\\ 
|x_{i_k}| & \textup{ if } j >k. \\ 
\end{array} \right.
\end{eqnarray} 
Also let $b_{(v)} := \| (v_{i_1}, v_{i_2}, v_{i_3}, \dots, v_{i_k} )\|_2$. The inequality 
\begin{eqnarray}
{v}^{\top} y \leq b_{(v)},\label{eq:cut}
\end{eqnarray} is a valid inequality for \ref{eq:SPCA}. 
\end{proposition}  
The validity of this inequality is clear: If $(x,y)$ is a feasible point, then the support of $y$ is at most $k$ and  $\| y \|_2 \leq 1$. Therefore, ${v}^{\top} y \leq \| (v_{i_1}, v_{i_2}, v_{i_3}, \dots, v_{i_k} )\|_2 = b_{(v)}$. Notice that this inequality is not valid for \ref{eq:ell_1-relax}.  Also see~\cite{phdthesisKim}.

\textcolor{black}{We add these inequalities at the end of each iteration for the model where the seeding $x$ for constructing $v$ is chosen to be the optimal solution of the previous iteration.
}

\section{Proof of Theorem \ref{thm:app-ratio}}\label{sec:app1}
Given a vector $v\in \mathbb{R}^n$, we denote the $j^{th}$ coordinate of $v$ as $v_j$, and for some $J \subseteq [n]$ we denote the projection of $v$ onto the coordinates in the index set $J$ as $v_J$. Define
\begin{align}
	S_k & \triangleq  \{x \in \mathbb{R}^n\,|\, \|x\|_2 \leq 1, \|x\|_0 \leq k\}, \\
	T_k & \triangleq \{x \in \mathbb{R}^n\,|\,\|x\|_2 \leq 1, \|x\|_1 \leq \sqrt{k} \}.
\end{align}

Note that any $x \in T_k$ can be represented as a nonnegative combination of points in $S_k$, i.e., $x = x^1 + \cdots + x^m$ and $x^i \in S_k$ for all $i$. Here we think of each $x^i$ as a projection onto some unique $k$ components of $x$ and setting the other components to $zero$.  Let $y^i = \frac{x^i}{\|x^i\|_2}$, then $y^i \in S_k$.  Now we have, $x = \sum_{i = 1}^m \|x^i\|_2 \cdot y^i$, and therefore
\begin{align}
	\frac{1}{\sum_{i = 1}^m \|x^i\|_2} x = \sum_{i = 1}^m \frac{\|x^i\|_2}{\sum_{i = 1}^m \|x^i\|_2} \cdot y^i.
\end{align}

Thus, if we scale $x \in T_k$ by $\|x^1\|_2 + \ldots + \|x^m\|_2$, then the resulting vector belongs to $\textup{conv}(S_k)$. Since we want this scaling factor to be as small as possible, we solve the following optimization problem:
\textcolor{black}{\begin{align}
	\min \|x^1\|_2 + \ldots + \|x^m\|_2: ~ x = x^1 + \ldots+ x^m; ~ x^i \in S_k, \forall i \in [m]. \tag{Bound} \label{eq:bound}
\end{align}}
%\textcolor{red}{I think there is a "+" before $x^m$.}

Without loss of generality, we assume that $x \geq 0$ and $x_1 \geq x_2 \geq \cdots \geq x_n \geq 0$. Let $x = \bar{v}^1 + \ldots + \bar{v}^m$ where $v^1, \ldots, v^m \in S_k$ is an optimal solution of \ref{eq:bound}. The following proposition presents a  result on an optimal solution of \ref{eq:bound}. 
\begin{proposition} \label{claim:structure-of-sol}
Let $I^1, \ldots, I^m$ be a collection of supports such that: $I^1$ indexes the $k$ largest (in absolute value) components in $x$, $I^2$ indexes the second $k$ largest (in absolute value) components in $x$, and so on (note that $m = \lceil \frac{n}{k} \rceil$). Then $I^1, \ldots, I^m$ is an optimal set of supports for \ref{eq:bound}.
\end{proposition}
%\begin{proof}
\proof{Proof.}
	We prove this result by the method of contradiction. Suppose we have an optimal representation as $x = \bar{v}^1 + \cdots \bar{v}^m$ --- and without loss of generality, we assume that $\|\bar{v}^1\|_2 \geq \cdots \geq \|\bar{v}^m\|_2$. Let $\bar{I}^1, \ldots, \bar{I}^m$ be the set of supports of $\bar{v}^1, \ldots, \bar{v}^m$ respectively, where we assume that the indices within each support vector are ordered such that $$(x_{\bar{I}^j})_1 \geq (x_{\bar{I}^j})_2 \geq \dots \geq (x_{\bar{I}^j})_g$$ for all $j \in \{1, \dots, m\}$ (note that $g = k$ if $j < m$).

Let $\bar{I}^p$ be the first support that is different from $I^p$, i.e., $\bar{I}^1 = I^1, \ldots, \bar{I}^{p - 1} = I^{p - 1}$ and $\bar{I}^p \neq I^p$. Let $I^p_{q}$ be the first index in $I^p$ that does not belong to $\bar{I}^p$ with $q \leq k$ since $\|\bar{I}^p\|_0 = k$. Therefore, $I^p_q$ must be in $\bar{I}^{p'}$ where $p' > p$. \textcolor{black}{Note now that by construction of $I$ and our assumption on $\bar{I}$, we have that $(x_{{I}^p})_q \geq (x_{\bar{I}^p})_q \geq (x_{ {\bar{I}}^p})_k$. Now we exchange the index $I^p_q$ in $\bar{I}^{p'}$ with $\bar{I}^p_k$ in $\bar{I}^p$.} We have:
	\begin{align}
		\sqrt{ \|x_{\bar{I}^p}\|_2^2 + ((x_{I^p})_q)^2 - ((x_{\bar{I}^p})_k)^2 } + \sqrt{ \|x_{\bar{I}^{p'}}\|_2^2 + ((x_{\bar{I}^p})_k)^2 - ((x_{I^p})_q)^2} \leq \|x_{\bar{I}^p}\|_2 + \|x_{\bar{I}^{p'}}\|_2,
	\end{align}
which holds because $\|x_{\bar{I}^p}\|_2 \geq  \|x_{\bar{I}^{p'}}\|_2$ and $((x_{I^p})_q)^2 - ((x_{\bar{I}^p})_k)^2\geq 0$.

Now repeating the above step, we obtain the result. 
%\end{proof}
\Halmos \endproof
Based on Proposition \ref{claim:structure-of-sol}, for any fixed $x \in T_k$, we can find out an optimal solution of \ref{eq:bound} in closed form. 
Now we would like to know, for which vector $x$, the scaling factor $\|v^1\|_2 + \ldots + \|v^m\|_2$  will be the largest.  Let $\rho$ be obtained by solving the following optimization problem:
\begin{align}
	\begin{array}{lllll}
		\rho = \max_{x} & \|x_{I^1}\|_2 + \cdots + \|x_{I^m}\|_2 \\
		\text{s.t.} & x = x_{I^1} + \cdots + x_{I^m} \\
		& \|x\|_2^2 = \|x_{I^1}\|_2^2 + \cdots + \|x_{I^m}\|_2^2 \leq 1 \\
		& \|x\|_1 = \|x_{I^1}\|_1 + \cdots + \|x_{I^m}\|_1 \leq \sqrt{k} \\
		& x_1 \geq \cdots \geq x_n \geq 0.
	\end{array} \tag{Approximation ratio} \label{eq:app-ratio}
\end{align} 
Then we obtain
\begin{align}
	T_k \subseteq \rho \cdot \text{Conv}\left( S_k \right).
\end{align}

Although the optimal objective value of  \ref{eq:app-ratio} is hard to compute exactly, we can still find an upper bound.

\begin{lemma} \label{lemma:app-ratio}
	The objective value $\rho$ of \ref{eq:app-ratio} is bounded from above by $1 + \sqrt{\frac{k}{k + 1}}$. 
\end{lemma}
%\begin{proof}
\proof{Proof.}
	First consider the case when $n \leq 2k$. In this case, $m \leq 2$. Consider the optimization problem:
\begin{eqnarray*}
\theta = \max & u + v\\
\text{s.t.} & u^2 + v^2 \leq 1 
\end{eqnarray*}
If we think of $\|x_{I^1}\|_2$ as $u$ and $\|x_{I^2}\|_2$ as $v$, then we see that  the above problem is a relaxation of \ref{eq:app-ratio} and therefore $\theta = \sqrt{2}$ is an upper bound on $\rho$. Noting that $\sqrt{2} \leq 1 + \sqrt{\frac{k}{k + 1}}$ for all $k \geq 1$, we have the result.  

Now we assume that $n > 2k$ and consequently $m > 2$.  From \ref{eq:app-ratio}, let $\|x_{I^1}\|_1 = t$ and $\|x_{I^1}\|_2 = \gamma$. Based on the standard relationship between $\ell_1$ and $\ell_2$ norm, we have $$\gamma \leq t \leq \sqrt{k} \gamma.$$
Since each coordinate of $x_{I^2}$ is smaller in magnitude than the average coordinate of $x_{I^{1}}$, we have
\begin{eqnarray}
\|x_{I^2}\|_2 \leq \sqrt{\left( \frac{\|x_{I^2}\|_1}{k}\right)^2 k} = \frac{t}{\sqrt{k}}. \label{eq:key2to1}
\end{eqnarray}
Also note that an alternative bound is given by $$\|x_{I^2}\|_2 \leq \sqrt{1 - \gamma^2}.$$ Using an argument similar to the one used to obtain (\ref{eq:key2to1}), we obtain that $$\sum_{i = 3}^m \|x_{I^{i}}\|_2\leq \sum_{i = 2}^{m-1} \sqrt{\left( \frac{\|x_{I^i}\|_1}{k}\right)^2 k}= \frac{1}{\sqrt{k}} \sum_{i = 2}^{m-1}\|x_{I^i}\|_1 \leq \frac{\sqrt{k} - t}{\sqrt{k}}.$$ Therefore we obtain
	\begin{align}
		\sum_{i = 1}^m \|x_{I^i}\|_2 & = \|x_{I^1}\|_2 + \|x_{I^2}\|_2 + \sum_{i = 3}^m \|x_{I^i}\|_2 \leq \gamma + \min\left\{ \frac{t}{\sqrt{k}}, \sqrt{1 - \gamma^2} \right\} + 1 - \frac{t}{\sqrt{k}}. \tag{Upper-Bound} \label{eq:app-UB}
	\end{align}
	Now we consider two cases:
	\begin{enumerate}
		\item If $\frac{t}{\sqrt{k}} \geq \sqrt{1 - \gamma^2}$, then \ref{eq:app-UB} becomes $\gamma + \sqrt{1 - \gamma^2} + 1 - \frac{t}{\sqrt{k}} $. Since $\gamma \geq \frac{t}{\sqrt{k}} \geq \sqrt{1 - \gamma^2}$, $\gamma$ satisfies $\gamma \geq \frac{1}{\sqrt{2}}$. Moreover we have that $t \geq \gamma, t \geq \sqrt{k(1 - \gamma^2)}$. Since $\gamma \leq \sqrt{k(1 - \gamma^2)}$ iff $\gamma \leq \sqrt{\frac{k}{k + 1}}$ we obtain two cases:
			\begin{align}
				\gamma + \sqrt{1 - \gamma^2} + 1 - \frac{t}{\sqrt{k}} & \leq \left\{
				\begin{array}{lll}
					\gamma + \sqrt{1 - \gamma^2} + 1 - \sqrt{1 - \gamma^2} & \text{ if } \gamma \in  \left[\frac{1}{\sqrt{2}}, ~ \sqrt{\frac{k}{k + 1}} \right] \\
					\gamma + \sqrt{1 - \gamma^2} + 1 - \frac{\gamma}{\sqrt{k}} & \text{ if } \gamma \in  \left[\sqrt{\frac{k}{k + 1}}, ~ 1 \right]
				\end{array}
				\right. \notag \\
				& \leq \left\{
				\begin{array}{lll}
					1 + \sqrt{\frac{k}{k + 1}}  \\
					1 + \sqrt{\frac{k}{k + 1}} 
				\end{array}
				\right. 
			\end{align}
			where (i) the first inequality holds when $\gamma = \sqrt{\frac{k}{k + 1}}$, (ii) the second inequality holds since the function $f(\gamma) = \gamma + \sqrt{1 - \gamma^2} + 1 - \frac{\gamma}{\sqrt{k}}$ achieves (local and global) maximum at point $\gamma = \sqrt{\frac{k + 1 - 2\sqrt{k}}{2k + 1 - 2\sqrt{k}}}$ which is less than $\sqrt{\frac{k}{k + 1}}$ for $k = 1, 2, \ldots$, thus $f(\gamma) \leq \max \left\{f\left(\sqrt{\frac{k}{k + 1}} \right), f(1) \right\} = 1 + \sqrt{\frac{k}{k + 1}}$ for part $\gamma \in  \left[\sqrt{\frac{k}{k + 1}}, ~ 1 \right]$.
		\item If $\frac{t}{\sqrt{k}} \leq \sqrt{1 - \gamma^2}$, then  \ref{eq:app-UB} becomes $\gamma + 1$. Note now that $\frac{\gamma}{\sqrt{k}} \leq \frac{t}{\sqrt{k}}\leq \sqrt{1 - \gamma^2}$, implies that $\gamma$ satisfies $\gamma \leq \sqrt{\frac{k}{k + 1}}$. Therefore, $1 + \gamma \leq 1 + \sqrt{\frac{k}{k + 1}}$.  
	\end{enumerate}
	Therefore, this upper bound holds. 
%\end{proof} 
\Halmos \endproof
Therefore, we can show Theorem \ref{thm:app-ratio} holds.
\proof{Proof of Theorem \ref{thm:app-ratio}.}
Since $T_k \subseteq \rho \cdot \text{Conv}\left( S_k \right)$ with $\rho \leq 1 + \sqrt{\frac{k}{k + 1}}$ and the objective function is maximizing a convex function, we obtain that $\lambda^k(A) \leq \text{OPT}_{\ell_1} \leq \rho^2 \cdot \lambda^k(A)$. 
\Halmos \endproof

\section{Numerical experiments} \label{sec:num-exp}
In this section, we report results on our empirical comparison of the performances of \ref{eq:convex-IP} method, \ref{eq:pert-convex-IP} method and the \ref{eq:SDP} relaxation method. 

\subsection{Hardware and Software}
All numerical experiments are implemented on MacBookPro13 with 2 GHz Intel Core i5 CPU and 8 GB 1867 MHz LPDDR3 Memory. Convex-IPs were solved using Gurobi 7.0.2. SDPs were solved using Mosek 8.0.0.60.

\subsection{Obtaining primal solutions} \label{sec:SPCA-heuristic}
We used a heuristic, which is very similar to the truncated power method~\cite{yuan2013truncated}, but has some advantages over the truncated power method. Given $v \in \mathbb{R}^n$, let $I_k(v)$ be the set of indices corresponding to the top $k$ entries of $v$ (in absolute value).

We start with a random initialization $x^0$ such that $\|x^0\|_2 = 1$, and set $I^0 \gets I_k(V^{\top}x^0)$ where $V$ is a square root of $A$, i.e. $A = V^{\top}V$. In the 
$i^{\text{th}}$ iteration, we update  
		\begin{align}
			I^i \gets I_k(V^{\top} x^i), ~ x^{i + 1} \gets {\color{black}\argmax_{\|x\|_2 = 1} }~ x^{\top} A_{I^i} x 
		\end{align}
		where $A_I \in \mathbb{R}^{n \times n}$ is the matrix with $[A_I]_{i,j} = [A]_{i,j}$ for all $i,j \in I$ and $[A_I]_{i,j} = 0$ otherwise. It is easy to see that $x^1, x^2, \ldots$ satisfy the condition $\|x\|_0 \leq k$. Moreover, using the fact $A$ is a PSD matrix, it is easy to verify that $(x^{i +1})^{\top} A x^{i+1} \geq (x^{i})^{\top} A x^{i}$ for all $i$. Therefore, in each iteration, the above heuristic method leads to an improved feasible solution for the \ref{eq:SPCA} problem. 

Our method has two clear advantages over the truncated power method:
\begin{itemize}
\item We use standard and efficient numerical linear algebra methods to compute eigenvalues of small $k\times k$ matrices. 
\item The termination criteria used in our algorithm is also simple: if  $I^i = I^{i'}$ for some $i' < i$, then we stop. Clearly, this leads to a finite termination criteria.
\end{itemize}

\textcolor{black}{In practice, we stop using a stopping criterion based on improvement and number of iterations instead of checking $I^i = I^{i'}$. Details are presented in Algorithm \ref{algo:heuristic}.}
	
\begin{algorithm}
\caption{Primal Algorithm} \label{algo:heuristic}
\begin{algorithmic}[1]
\State \emph{Input}: Sample covariance matrix $A$, cardinality constraint $k$, initial vector $x^0$.
\State \emph{Output}: A feasible solution $x^{\ast}$ of \ref{eq:SPCA}, and its objective value. 
\Function{Heuristic Method}{$A, k, x^0$}\label{function:heuristic}
\State Start with an initial (randomized) vector $x^0$ such that $\|x^0\|_2 = 1$ and $\|x^0\|_0 \leq k$. 
\State Set the initial current objective value $\text{Obj} \gets (x^0)^{\top} A x^0$.
\State Set the initial past objective value $\tilde{\text{Obj}}\gets 0$.
\State Set the maximum number of iterations be $i^{\max}$. 
\While{$\text{Obj} - \tilde{\text{Obj}} > \epsilon$ and $i \leq i^{\max}$}
\State Set $\tilde{\text{Obj}} \gets \text{Obj}$.
\State Set $I^i \gets I_k(V^{\top} x^i)$.
\State Set $x^{i + 1} \gets \arg\max_{\|x\|_2 = 1} x^{\top} A_{I^i} x$.
\State Set $\text{Obj} \gets (x^{i + 1})^{\top} A x^{i + 1}$.
\EndWhile
\State \Return $x^{\ast}$ as the final $x$ obtained from while-loop, and \text{Obj}.
\EndFunction
\end{algorithmic}
\end{algorithm}
We use the values of $\epsilon = 10^{-6}$ and $i^{\max} = 20$ in our experiments in Algorithm~\ref{algo:heuristic}. We repeat this algorithm with multiple random initializations. We repeat 20 times and take the best solution. We emphasize that Algorithm \ref{algo:heuristic} may not lead to a global solution of \ref{eq:SPCA}.

Our Algorithm may also be interpreted as a version of the ``alternating method" used regularly as a heuristic for bilinear programs as the sparse PCA problem can be equivalently rewritten as $\textup{max} \{x^{\top}Ay \,|\, \|x\|_2 = \|y\|_2 = 1, \|x\|_0 \leq k, \|y\|_0 \leq k\}$. We have compared our primal method to two standard heuristics for finding primal feasible solutions of the sparse PCA problems in the literature: truncated power method (TPM, \cite{yuan2011truncated}), generalized power method (GPM, \cite{journee2010generalized}) with $\ell_0$-penalty. The performances of all these methods are quite similar to our method (in terms of primal objective function values) on the real instances; see details in Appendix~\ref{sec:LB-comparison}.

\subsection{Implementation of \ref{eq:convex-IP} model and \ref{eq:pert-convex-IP} model}
\subsubsection{Deciding $\lambda$, $N$}
\begin{enumerate}	
	\item Deciding $\lambda$: The size of the set $\{i: \lambda_i > \lambda\}$ denoted by $I_{\text{pos}}$ plays an important role for the computational tractability of our method. \textcolor{black}{So our algorithm inputs an initial value, $I^{\text{ini}}_{\text{pos}}$ }. From the primal heuristic, we obtain a lower bound $\text{LB}^{\text{primal}}$ on $\lambda^k(A)$.  Let 
$$\lambda_{i_1} \geq \lambda_{i_2} \geq \dots \geq \lambda_{i_n}$$
be the eigenvalues of $A$. If $\lambda_{i_{I^{\text{ini}}_{\text{pos}}}} < \text{LB}^{\text{primal}}$, then we set $\lambda \triangleq \lambda_{i_{I^{\text{ini}}_{\text{pos}}}}$. On the other hand, if $\lambda_{i_{I^{\text{ini}}_{\text{pos}}}} > \text{LB}^{\text{primal}}$, then \textcolor{black}{let $l$ be the smallest index} such that $\lambda_{i_{l}} > \text{LB}^{\text{primal}}$ and we set $\lambda \triangleq \lambda_{i_{l }}$.

	\item Deciding $N$: In practice, $\theta_i$ was found to be significantly smaller than $1$. So we used a value of $N =3 $ in all our experiments. 
\end{enumerate}
\subsubsection{Final details}
A total time of $7200$ seconds were given to each instance for running the convex IP (any extra time reported in the tables is due to running time of singular value decomposition and primal heuristics). We have run all our experiments with $k = 10, 20$. For the \ref{eq:convex-IP} method, we use: $(I^{\text{ini}}_{\text{pos}}, N) = (10, 3)$. For the \ref{eq:pert-convex-IP} method, we 
let ``$\text{iter}$'' denote the maximum number of iterations. We used three settings in our experiments:
\begin{align*}
	(I^{\text{ini}}_{\text{pos}}, N, \text{iter})  \in \left\{(5, 3, 10), ~ (10, 3, 3), ~ (15, 3, 2)  \right\}.
\end{align*}

The overall algorithms using the \ref{eq:pert-convex-IP} model and the \ref{eq:convex-IP} model are presented in Appendix~\ref{sec:appconIP}.

\subsection{Data Sets}
We conduct numerical experiments on two types of data sets. Details of these two types of data sets are presented in Appendix~\ref{sec:data-set}. 
\begin{itemize}
	\item \textbf{Artificial data set:} Tables~\ref{tab:toy-example}, \ref{tab:toy-example-20}, \ref{tab:artificial-data},  \ref{tab:artificial-data-20}, \ref{tab:controlling-sparsity}, \ref{tab:controlling-sparsity-20} present results for artificial/synthetic datasets. 
	\item \textbf{Real data set:} Tables~\ref{tab:pitprops}, \ref{tab:bio-internet-data}, \ref{tab:bio-internet-data-20} show results for real data sets.
\end{itemize}

\subsection{Description of the rows/columns in the tables}
Note that the labels for each of the columns in Tables~\ref{tab:toy-example}, \ref{tab:toy-example-20}, \ref{tab:artificial-data},  \ref{tab:artificial-data-20}, \ref{tab:controlling-sparsity}, \ref{tab:controlling-sparsity-20}, \ref{tab:pitprops}, \ref{tab:bio-internet-data}, \ref{tab:bio-internet-data-20} are as follows:
\begin{itemize}
	\item \textbf{Case:} The first part is a name. `\textbf{Case 1}' or `\textbf{Case 2}' denotes the instance number. The second part is the format $(\text{size}, \text{cardinality})$ which denotes the number of columns/rows of the $A$ matrix and the right-hand-side of the $\ell_0$ constraint of the original \ref{eq:SPCA} problem. 
	\item \textbf{LB-$\ell_0$:} denotes the lower bound on the \ref{eq:SPCA} problem obtained from the (heuristic) Algorithm~\ref{algo:heuristic} in Section \ref{sec:SPCA-heuristic}.
	\item \textbf{\#-$\lambda$:} denotes the size of set $\{i \,|\, \lambda_i > \text{LB-}\ell_0\}$ where $\lambda_i$ are the eigenvalues of the covariance matrix. 
	\item \textbf{Convex-IP-$\ell_0$, Pert-Convex-$\text{IP}_0$:} \textcolor{black}{denote the \ref{eq:convex-IP} and the \ref{eq:pert-convex-IP} models.}
	\item \textbf{SDP:} denotes the semidefinite programming relaxation  solved using Mosek. In Appendix~\ref{sec:UB-comparison}, we compare the dual bounds by alternative methods \cite{de2018using} to solve the SDP-relaxation for the real instances. Our conclusion based on our  implementation of other algorithms is that when Mosek solves the instance, the best dual bound is obtained from Mosek. For some slightly larger instances, other algorithms might produce dual bounds. Usually, these dual bounds are extremely poor in quality. Moreover, these other methods do not scale up to instances with $d \geq 1000$. Therefore, we have chosen to present results only from Mosek in Tables~\ref{tab:toy-example}, \ref{tab:toy-example-20}, \ref{tab:artificial-data},  \ref{tab:artificial-data-20}, \ref{tab:controlling-sparsity}, \ref{tab:controlling-sparsity-20}, \ref{tab:pitprops}, \ref{tab:bio-internet-data}, \ref{tab:bio-internet-data-20}; and the remaining results are relegated to Appendix~\ref{sec:UB-comparison}.
	\item \textbf{UB:} denotes the upper bound obtained from current dual bound method (i.e., Convex-IP-$\ell_0$, Pert-Convex-$\text{IP}_0$, SDP). 
	\item \textbf{gap:} denotes the approximation ratio (duality gap) obtained by the formula $\textbf{gap}= \frac{\text{UB} - \text{LB-}\ell_0}{\text{LB-}\ell_0}$. 
	\item \textbf{time:} denotes the total running time---we present the overall running time due to singular value decomposition, heuristic method to obtain primal solutions, and solvers (Gurobi, Mosek) used to solve integer programming (set to terminate within 7200 seconds). 
\end{itemize}

The three rows  corresponding to Pert-Convex-IP, corresponds to experiments with three settings: 
$(I_{\text{pos}}, N, \text{iter}) = \left\{ (5, 3, 10), ~ (10, 3, 3), ~ (15, 3, 2) \right\}. $

\subsection{Conclusions and summary of numerical experiments}
Based on numerical results reported in Tables~\ref{tab:toy-example}, \ref{tab:toy-example-20}, \ref{tab:artificial-data},  \ref{tab:artificial-data-20}, \ref{tab:controlling-sparsity}, \ref{tab:controlling-sparsity-20}, \ref{tab:pitprops}, \ref{tab:bio-internet-data}, \ref{tab:bio-internet-data-20} we draw some preliminary observations:
\begin{enumerate}
\item \textbf{Size of instances solved:}
\begin{itemize}
\item SDP: Because of limitation of hardware and software, the SDP relaxation method does not solve instances with input matrix of size greater than or equal to $300 \times 300$. 
\item Convex-IP: The convex IP shows better scalability than the SDP relaxation and produces dual bounds for instances with input matrix of size up to $500 \times 500$. 
\item Pert-Convex-IP: The perturbed convex IP scales significantly better that the other methods. While we experimented with instances up to size $2000 \times 2000$, we believe this method will easily scale to larger instances, when $k = 10, 20$ with $(I_{\text{pos}}, N)$ being chosen appropriately.       
\end{itemize} 
\item \textbf{Quality of dual bound:}
\begin{itemize}
\item SDP vs Best of $\{$Convex-IP, Pert-Convex-IP$\}$: While on some instances SDP obtained better dual bounds, this was not the case for all instances. For example, on the `controlling sparsity' random instances and both the real data sets Eisen-1 and Eisen-2, SDP bounds are weaker. 
\item Convex-IP vs Pert-Convex-IP: If the convex IP solved within the time limit, then usually the bound is better than that obtained for Pert-Convex-IP. In other cases,  Pert-Convex-IP performs better as it is easy to solve and usually solves within 1 hour.  
\item Overall gaps for Best of $\{$Convex-IP, Pert-Convex-IP$\}$: Except for the random instances of type `controlling sparsity' of size $1000\times 1000$, and Lymphoma data set, in all other instances at least one method had a gap less that $10\%$. 
\item Cardinality 10 vs Cardinality 20: When the cardinality budget is allowed to increase, based on our numerical results, we can see that the running time of our \ref{eq:convex-IP} and \ref{eq:pert-convex-IP} methods do not change a lot, since the parameter of cardinality $k$ of \ref{eq:convex-IP} and \ref{eq:pert-convex-IP} method only influences the linear constraint $\sum_{i = 1}^n y_i \leq \sqrt{k}$, which is more robust to changes in the value of the cardinality $k$ than typical cardinality constraint in interger programming.
\end{itemize}

\item {\color{black}\textbf{Comparison of different numbers of splitting points (parameter $N$):} We compare the performances of the Pert-Convex-$\text{IP}_0$ method under distinct initialization splitting points with $(I_{\text{pos}}, N_{\text{ini}}, \text{\# of iterations}) = (5,1,1), (5,3,1), (5,5,1)$, see Table~\ref{tab:comparison-bio}. We present results with just one round of iterations to clearly understand the effect of number of splitting points. We observe that the gap decreases when the number of splitting points increases. On the other hand, the running time increases with the number of splitting points incereasing. However increasing splitting points from 3 to 5 does not significantly improve the bounds.}

\begin{table}[ht] 
\caption {Comparison of distinct splitting points} \label{tab:comparison-bio}
\begin{center}
\resizebox{1 \textwidth}{!}{\begin{tabular}{|l|l|ll|ll|ll|l|l|l|l|l|} \hline
	
\multirow{2}*{\bf Instance $\backslash$ Splitting points }& \multirow{2}*{LB} & \multicolumn{2}{|l|}{\bf $(5,1,1)$} & \multicolumn{2}{|l|}{\bf $(5,3,1)$} & \multicolumn{2}{|l|}{\bf $(5,5,1)$}\\ \cline{3 - 8} 
& & gap & Time & gap & Time & gap & Time \\ \specialrule{.15em}{.05em}{.05em} 
\bf Eisen-1 (79, 10) & 17.335 & 2.619 \% & 2.762 &  0.588 \% & 3.049 & \bf 0.329 \% & 3.127  \\ \hline
\bf Eisen-2 (118, 10) & 11.718 &  13.245 \% & 5.738 &  4.736 \% & 7.194 &  \bf 4.207 \% & 7.78 \\ \hline
\bf Colon (500, 10) & 2641.229 & 30.652 \% & 72.802 & 27.755\% & 73.149 & \bf 27.673 \% & 76.115 \\ \hline 
\bf Lymphoma (500, 10) & 6008.741 &  52.412 \% & 95.561 & 43.956 \% &  83.902 & \bf 43.587 \% & 86.422 \\ \hline
\bf Reddit (2000, 10) & 1052.934 &  8.548 \% & 1628.128 & 4.136 \% & 1450.775 & \bf 3.999 \% & 1488.936 \\ \hline 
\end{tabular}}
\end{center}
\end{table}

{\color{black}
\item \textbf{Comparison between $\ell_1$-relaxation and original sparsity constraint:} To further illustrate why we prescribe the use of $\ell_1$ relaxation to obtain dual bounds of SPCA, we compare the following two models: (1) The \ref{eq:pert-convex-IP} model used in the paper; (2) The same ``perturbed convex IP" where the $\ell_1$ constraint is replaced by a cardinality constraint (with the introduction of binary variables), denoted as \ref{eq:pert-convex-IP-0}. 
\begin{align*}
	\begin{array}{rllll}
		\max & ~ {\lambda} + \sum_{i \in \{i : \lambda_i > {\lambda}\}} (\lambda_i - {\lambda}) \xi_i - s & \\
		\text{s.t. } & \left\{
		\begin{array}{lll}
		g_i = x^{\top} v_i \\
		- \theta_i \leq g_i \leq \theta_i
		\end{array}
		\right. & i \in \{i: \lambda_i > {\lambda} \} \\
		& \left\{
		\begin{array}{lll}
		g_i = \sum_{j = - N}^N \gamma_i^j \eta_i^j \\
		\xi_i = \sum_{j = - N}^N (\gamma_i^j)^2\eta_i^j \\
		(\eta_i^{-N}, \ldots, \eta_i^N) \in \text{SOS-2} 
		\end{array}
		\right. & i \in \{i: \lambda_i > {\lambda} \} \\
		& \left\{
		\begin{array}{lll}
		\sum_{i = 1}^n x_i^2 \leq 1 \\
		\sum_{i \in \{i: \lambda_i > {\lambda}\}} g_i^2 \leq 1 - \frac{s}{{\lambda} - \bar{\lambda}} \\
		1 - \frac{s}{{\lambda} - \bar{\lambda}} \leq \sum_{i \in \{i: \lambda_i > {\lambda}\}} \xi_i \leq 1 + \sum_{i \in \{i: \lambda_i > {\lambda}\}}\frac{\theta_i^2}{4 N^2} - \frac{s}{{\lambda} - \bar{\lambda}}
		\end{array}
		\right. & \\
		& \left\{
		\begin{array}{lll}
		\sum_{i = 1}^n z_i \leq k\\
		z_i \geq x_i, z_i \geq - x_i, z_i \in \{0, 1\}, \forall i \in [n] \\
	%	y_i \geq x_i, ~ y_i \geq - x_i, ~ \forall i \in [n]
		\end{array} 
		\right. & \text{($\ell_0$ constraint)} \
%		& ~ v^{\top} y \leq b_{(v)}
	\end{array} \tag{Model-with-$\ell_0$} \label{eq:pert-convex-IP-0}
\end{align*}
We tested on the real-life data for $k = 10$ and $k =20$ in Table~\ref{tab:real-10-l1-l0-comparison}, Table~\ref{tab:real-20-l1-l0-comparison}. All parameters $(I_{\text{pos}}, N_{\text{ini}}, \# \text{iter})$ are also listed in Table~\ref{tab:real-10-l1-l0-comparison}, Table~\ref{tab:real-20-l1-l0-comparison} which are the same as the parameters that used in the Section 4.3.2 (except for $\# \text{iter} = 1$ here).  

\begin{table}[!h]
\caption{Comparison: Real Instances, cardinality parameter $k = 10$}
\label{tab:real-10-l1-l0-comparison} 
\begin{center} 
\resizebox{ 0.9\textwidth}{!}{\begin{tabular}{|l|l|l|l|l|l|l|l|l|l|l|l|l|l|l|l|} \hline

\multirow{2}*{\bf (size, index)} & \multirow{2}*{\bf($I_{\text{pos}}$, $N_{\text{ini}}$, \# iter)}  & \multicolumn{2}{|l|}{\ref{eq:pert-convex-IP}} & \multicolumn{2}{|l|}{Model-with-$\ell_0$}  \\ \cline{3 - 6}
 & & Gap & Time & Gap & Time  \\ \hline 
 
 Eisen Data 1 (79) & (5, 3, 1) & 0.588 \% & 2.86 & \bf 0.392 \% & 8.591 \\ 
  & (10, 3, 1) & 0.796 \% & 3.863 &  0.525 \% & 99.168 \\
  & (15, 3, 1) & 0.865 \% & 10.049 & 0.588 \% & 685.519 \\ \hline
  
 Eisen Data 2 (118) & (5, 3, 1) & 4.736 \% & 6.576 &  4.48 \% & 86.251 \\ 
  & (10, 3, 1) & 2.364 \% & 27.525 &  2.321 \% & 2105.51 \\  
  & (15, 3, 1) & 1.997 \% & 195.356 & \bf 1.971 \% & 5935.205 \\ \hline 
 
 Matrix CovColon (500) & (5, 3, 1) & 27.755 \% & 90.362 &  4.48 \% & 86.251 \\ 
  & (10, 3, 1) & 2.364 \% & 27.525 & \bf 2.321 \% & 2105.51 \\ 
  & (15, 3, 1) & 5.349 \% & 2610.972 & 11.51 \% & 7288.835 \\  \hline 
  
 Matrix LymphomaCov (500) & (5, 3, 1) &  43.956 \% & 87.159 & 47.93 \% & 7305.024 \\ 
  & (10, 3, 1) &  23.662 \% & 355.236 & 39.431 \% & 7289.135 \\ 
  & (15, 3, 1) & \bf 17.863 \% & 4224.933 & 39.526 \% & 7309.047 \\ \hline 
  
 Reddit (2000) & (5, 3, 1) &4.136 \% & 1867.157 & 5.826 \% & 8765.165 \\ 
  & (10, 3, 1) & \bf 3.446 \% & 1831.221 & 8.867 \% & 8638.037 \\ 
  & (15, 3, 1) & 3.523 \% & 3726.841 & 10.356 \% & 8542.98 \\ \hline

\end{tabular}}
\end{center}
\end{table}

\begin{table}[!h]
\caption{Comparison: Real Instances, cardinality parameter $k = 20$} 
\label{tab:real-20-l1-l0-comparison} 
\begin{center} 
\resizebox{ 0.9\textwidth}{!}{\begin{tabular}{|l|l|l|l|l|l|l|l|l|l|l|l|l|l|l|l|} \hline 

\multirow{2}*{\bf (size, index)} & \multirow{2}*{\bf($I_{\text{pos}}$, $N_{\text{ini}}$, \# iter)}  & \multicolumn{2}{|l|}{\ref{eq:pert-convex-IP}} & \multicolumn{2}{|l|}{Model-with-$\ell_0$}  \\ \cline{3 - 6}
 & & Gap & Time & Gap & Time  \\ \hline 
 
 Eisen Data 1 (79) & (5, 3, 1) & \bf 0.559 \% & 3.183 & 1.298 \% & 7204.468 \\ 
  & (10, 3, 1) &  0.813 \% & 20.568 & 2.985 \% & 7204.059 \\ 
  & (15, 3, 1) &  0.886 \% & 1016.839 & 5.519 \% & 7229.677 \\ \hline  
  
 Eisen Data 2 (118) & (5, 3, 1) &  1.837 \% & 6.48 & 2.65 \% & 8062.349 \\ 
  & (10, 3, 1) & 1.18 \% & 46.001 & 4.223 \% & 7211.949 \\ 
  & (15, 3, 1) & \bf1.087 \% & 443.759 & 3.664 \% & 7205.331  \\ \hline 
  
 Matrix CovColon (500) & (5, 3, 1) &  17.014 \% & 75.267 & 18.539 \% & 7268.644 \\ 
  & (10, 3, 1) & 6.528 \% & 372.802 & 12.903 \% & 7271.37 \\  
  & (15, 3, 1) &  \bf6.066 \% & 7275.58 & 12.737 \% & 7273.013 \\ \hline 
  
 Matrix LymphomaCov (500) & (5, 3, 1) & 24.042 \% & 91.786 & 26.622 \% & 7288.825 \\ 
  & (10, 3, 1) & 14.498 \% & 214.784 & 24.381 \% & 7302.236 \\ 
  & (15, 3, 1) & \bf11.811 \% & 3349.161 & 35.286 \% & 8831.009 \\ \hline 
 
 Reddit (2000) & (5, 3, 1) & \bf4.286 \% & 4652.869 & 7.139 \% & 8708.004 \\ 
  & (10, 3, 1) & 4.288 \% & 1677.933 & 9.647 \% & 8546.823 \\ 
  & (15, 3, 1) & 4.776 \% & 4274.327 & 12.157 \% & 8560.558 \\ \hline 

\end{tabular}}
\end{center}
\end{table}

Based on the Table~\ref{tab:real-10-l1-l0-comparison} \ref{tab:real-20-l1-l0-comparison}, following conclusions can be obtained: 
\begin{enumerate}
	\item For instances with relative small size ($\leq 500$): the upper bounds (UB) obtained from \ref{eq:pert-convex-IP-0} is a slightly better than the upper bounds (UB) from \ref{eq:pert-convex-IP}, but the running time used for \ref{eq:pert-convex-IP-0} is much longer than \ref{eq:pert-convex-IP}. 
	\item For instances with relative large size ($\geq 500$): both the upper bounds and the running time obtained from \ref{eq:pert-convex-IP} method are significantly better than those  obtained from \ref{eq:pert-convex-IP-0}. In another words, the \ref{eq:pert-convex-IP} is more scalable. 
\item Effect of $k$: We see that for $k =20$ the performance of \ref{eq:pert-convex-IP} method is even more dramatically better than that of \ref{eq:pert-convex-IP-0}. In fact, now \ref{eq:pert-convex-IP} beats \ref{eq:pert-convex-IP-0} on quality of bound and time even for small ($\leq 500$) instances. Indeed, this is another nice property of the $\ell_1$-relaxation, namely it handles larger values of $k$ more robustly.
\end{enumerate}
}
\end{enumerate}

\begin{table}[H]
\caption {Spiked Covariance Recovery - Cardinality 10} \label{tab:toy-example}
\begin{center}
\resizebox{0.9 \textwidth}{!}{\begin{tabular}{|l|l|l|l|l|l|l|l|l|l|l|l|l|l|l|l|l|l|} \hline
	
\multirow{2}*{\bf Case }& \multirow{2}*{\bf LB-$\ell_0$}& \multirow{2}*{\bf \#-$\lambda$} & \multicolumn{2}{|l|}{\bf Convex-IP-$\ell_0$} & \multicolumn{2}{|l|}{\bf Pert-Convex-$\text{IP}_0$} & \multicolumn{2}{|l|}{\bf SDP} \\ \cline{4 - 9} 

& & &  gap & Time &  gap & Time & gap & Time \\ \hline

\bf Case 1 (200, 10) & 511.95 & 1 &  0.005 \% & 380 &   0.007 \% & 76 & \bf 0.001 \% & 1277 \\ 
\bf  &  &  &  &  &  0.005 \% & 230 & &    \\ 
\bf  &  &  &  &  &  0.005 \% & 1605 & &  \\
\hline
\bf Case 2 (200, 10) & 592.45 & 1 &  0.003 \% & 469 &   0.006 \% & 615 &  \bf 0.002 \% & 1458  \\
\bf  &  &  &  &  &   0.006  \% & 236 & &  \\ 
\bf  &  &  &  &  &  0.005 \% & 325 &  &   \\
\hline
\bf Case 1 (300, 10) & 414.04 & 1 &  \bf 0.027 \% & 1692 &  0.03 \% & 642 & NaN & - \\ 
\bf  &  &  & &  &    0.029 \% & 407 & & \\ 
\bf  &  &  & &  &    0.027 \% & 796 &  & \\
\hline
\bf Case 2 (300, 10) & 568.56 & 1 & \bf 0.011 \% & 1067 &  0.016 \% & 82 &NaN & -  \\ 
\bf  &  &  & &  &    0.014 \% & 493 & &\\ 
\bf  &  &  & &  &    0.012 \% & 942 & & \\
\hline
\bf Case 1 (400, 10) & 478.24 & 1 & \bf 0.025 \% & 2598 &   0.04 \% & 793 & NaN & -  \\ 
\bf  &  &  & &  &     0.03\% & 610 & & \\ 
\bf  &  &  & &  &    0.03\% & 1495 & & \\
\hline
\bf Case 2 (400, 10) & 426.91 & 1 & \bf 0.037 \% & 3374 &  0.06 \% & 181 & NaN & -  \\ 
\bf  &  &  & &  &   0.05 \% & 846 & & \\ 
\bf  &  &  & &  &   0.04 \% & 2137 & & \\
\hline
\bf Case 1 (500, 10) & 256.82 & 1 & \bf 0.164 \% & 7525 &  0.21 \% & 1345 & NaN & - \\ 
\bf  &  &  & &  &    0.18 \% & 1512 & & \\ 
\bf  &  &  & &  &   0.17 \% & 3279 & & \\
\hline
\bf Case 2 (500, 10) & 551.74 & 1 & \bf 0.029 \% & 7196 &  0.04 \% & 152 & NaN & - \\ 
\bf  &  &  & &  &   0.04 \% & 725 & &  \\ 
\bf  &  &  & &  &    0.03 \% & 1694 & & \\
\hline
\bf Case 1 (1000, 10) & 315.16 & 1 & NaN & - & 0.57 \% & 1147 & NaN & - \\ 
\bf  &  &  & &  &    0.52 \% & 776 & &  \\ 
\bf  &  &  & &  &    \bf 0.53 \% & 3633 & &  \\
\hline
\bf Case 2 (1000, 10) & 383.44 & 1 & NaN & - & 0.34 \% & 2745 & NaN & - \\ 
\bf  &  &  & &  &   \bf 0.32 \% & 403 &  &  \\ 
\bf  &  &  & &  &   0.34 \% & 3643 & & \\
\hline
\end{tabular}}
\end{center}
\end{table}

\pagebreak

\begin{table}[H]
\caption {Spiked Covariance Recovery - Cardinality 20} \label{tab:toy-example-20}
\begin{center}
\resizebox{0.9 \textwidth}{!}{\begin{tabular}{|l|l|l|l|l|l|l|l|l|l|l|l|l|l|l|l|} \hline
	
\multirow{2}*{\bf Case }& \multirow{2}*{\bf LB-$\ell_0$}& \multirow{2}*{\bf \#-$\lambda$} & \multicolumn{2}{|l|}{\bf Convex-IP-$\ell_0$} & \multicolumn{2}{|l|}{\bf Pert-Convex-$\text{IP}_0$} & \multicolumn{2}{|l|}{\bf SDP} \\ \cline{4 - 9} 

& & &  gap & Time & gap & Time & gap & Time \\ \hline

\bf Case 1 (200, 20) & 516.756 & 1 & 2.05 \% & 493 &  \bf 0.008 \% & 746 & - \% & -  \\ 
\bf  &  &  & &  & 0.073 \% & 3116 &  &  \\ 
\bf  &  &  & &  & 0.573 \% & 7214 &  &  \\
\hline

\bf Case 2 (200, 20) & 593.651 & 1 & 0.98 \% & 1847 &  \bf 0.005 \% & 323 & -\% & -\\ 
\bf  &  &  & &   &  0.006 \% & 5992 &  &   \\ 
\bf  &  &  & &   &  0.102 \% & 7215 &  &  \\
\hline

\bf Case 1 (300, 20) & 499.92 & 1 & 0.70 \% & 1848 &  \bf 0.018 \% & 745& -\% &- \\ 
\bf  &  & &  &  & 0.021 \% & 4799 &  &   \\ 
\bf  &  & &  &  &  0.399 \% & 7230 &  &  \\
\hline

\bf Case 2 (300, 20) & 600.553 & 1 & 1.13 \% & 1771 &  0.014 \% & 530 & -\% & - \\ 
\bf  &  & &  &  &  \bf 0.013 \% & 2964&  &  \\ 
\bf  &  & &  &  &  0.272 \% & 7232 &  &  \\
\hline

\bf Case 1 (400, 20) & 483.995 & 1 & 2.74 \% & 6398 &  \bf 0.034 \% & 1186 & -\% -& \\ 
\bf  &  & & &  &  0.168 \% & 7262 &  &   \\ 
\bf  &  & & &  &  0.832 \% & 7255 &  &   \\
\hline

\bf Case 2 (400, 20) & 428.275 & 1 & 1.92 \% & 7426 &  \bf 0.045 \% & 576 & -\% & - \\ 
\bf  &  & &  &  &  0.074 \% & 6965 &  &   \\ 
\bf  &  & &  &  &  0.53 \% & 7251& - &  \\
\hline

\bf Case 1 (500, 20) & 294.35 & 1 & 1.19 \% & 7027 &  \bf 0.162 \% & 1341 & -\% & - \\ 
\bf  &  & &  &  &   0.165 \% & 6087&  &    \\ 
\bf  &  & &  &  &   1.285 \% & 7294 &  &    \\
\hline

\bf Case 2 (500, 20) & 571.15 & 1 & 1.96 \% & 4628 &  \bf 0.039 \% & 1862&- \% & - \\ 
\bf  &  &  & &  &   0.2 \% & 1935&  &   \\ 
\bf  &  &  & &  &   1.215 \% & 3360 &  &  \\
\hline

\bf Case 1 (1000, 20) & 414 & 1 & - \% & - &  0.53 \% & 3133 & - \% & - \\ 
\bf  &  & &  &  &  \bf 0.50 \% & 2760 &  &   \\ 
\bf  &  & &  &  &  \bf 0.50 \% & 5844 &  &   \\
\hline

\bf Case 2 (1000, 20) & 391.795 & 1 & - \% & - & \bf 0.311 \% & 4756 & -\% &- \\ 
\bf  &  & & &  &   0.74 \% & 3596 &    &  \\ 
\bf  &  & & &  &   2.906 \% & 7516&    &  \\
\hline

\end{tabular}}
\end{center}
\end{table}

\pagebreak

\begin{table}[H]
\caption {Synthetic Example - Cardinality 10} \label{tab:artificial-data}
\begin{center}
\resizebox{0.9 \textwidth}{!}{\begin{tabular}{|l|l|l|l|l|l|l|l|l|l|l|l|l|l|l|l|} \hline
	
\multirow{2}*{\bf Case }& \multirow{2}*{\bf LB-$\ell_0$} & \multirow{2}*{\bf \#-$\lambda$} & \multicolumn{2}{|l|}{\bf Convex-IP-$\ell_0$} & \multicolumn{2}{|l|}{\bf Pert-Convex-$\text{IP}_0$} & \multicolumn{2}{|l|}{\bf SDP}\\ \cline{4 - 9}

& & & gap & Time & gap & Time & gap & Time \\ \specialrule{.15em}{.05em}{.05em}  

	\bf Case 1 (200, 10) & 5634.143 & 3 & 11.884 \% & 7205 & 0.14 \% & 38 & \bf 0.10 \% & 1092  \\ 
	\bf & & & & & 0.15 \% & 16 &  &   \\
	\bf & & & & & 0.15 \% & 186 &  &  \\
	\hline 
	\bf Case 2 (200, 10) & 7321.23 & 3 & 1.703 \% & 7205 &  0.13 \% & 23 & \bf 0.09 \% & 1086  \\ 
	\bf & & & & &  0.13 \% & 13 &  &     \\
	\bf & & & & &  0.12 \% & 47 &  &     \\
	\hline 
	\bf Case 1 (300, 10) & 4157.46 & 3 & 51.072 \% & 7210 & \bf 0.27 \% & 83 & NaN & - \\ 
	\bf & & & & &  0.29 \% & 21 &  &     \\
	\bf & & & & &  0.27 \% & 486 &  &    \\
	\hline 
	 \bf Case 2 (300, 10) & 5135.50 & 3 & 65.275 \% & 7210  & 0.23 \% & 62 & NaN & - \\ 
	\bf & & & & &  \bf 0.22 \% & 59 &  &   \\
	\bf & & & & &  0.23 \% & 58 &  &   \\
	 \hline
	 \bf Case 1 (400, 10) & 6519.37 & 3 & 55.308 \% & 7219 &  \bf 0.22 \% & 98 & NaN & -  \\ 
	\bf & & & & &  0.23 \% & 23 &  &    \\
	\bf & & & & &  \bf 0.22 \% & 349 &  &   \\
	 \hline 
	 \bf Case 2 (400, 10) & 5942.05 & 3 &  45.396 \% & 7218 &  \bf 0.36 \% & 56 & NaN & - \\ 
	\bf & & & & &  0.42 \% & 29 &  &   \\
	\bf & & & & &  0.41 \% & 364 &  &  \\ 
	\hline
	 \bf Case 1 (500, 10) & 5125.86 & 3 & 65.98 \% & 7230 &  0.38 \% & 149 & NaN & - \\ 
	\bf & & & & &  0.38 \% & 44 &  &   \\
	\bf & & & & &  \bf 0.37 \% & 132 &   & \\
	\hline 
	 \bf Case 2 (500, 10) & 5545.85 & 3 & 48.328 \% & 7230 & 0.39 \% & 50 & NaN & - \\ 
	\bf & & & & & \bf  0.38 \% & 30 &  &   \\
	\bf & & & & & \bf 0.38 \% & 231 &  &  \\ 
	\hline
	\bf Case 1 (1000, 10) & 5116.08 & 3 & NaN & - &   0.58 \% & 257 & NaN & -  \\ 
	\bf & & & & &  \bf 0.57 \% & 128 &  &    \\
	\bf & & & & &  \bf 0.57 \% & 1373 &   & \\ 
	\hline
	\bf Case 2 (1000, 10) & 6946.12 & 3 & NaN & - &  0.39 \% & 323 & NaN & - \\ 
	\bf & & & & &  0.36 \% & 129 &  &   \\
	\bf & & & & &  \bf 0.34 \% & 1167 &  &  \\ 
	\hline
\end{tabular}}
\end{center}
\end{table}

\pagebreak

\begin{table}[H]
\caption {Synthetic Example- Cardinality 20} \label{tab:artificial-data-20}
\begin{center}
\resizebox{0.9 \textwidth}{!}{\begin{tabular}{|l|l|l|l|l|l|l|l|l|l|l|l|l|l|l|l|} \hline
	
\multirow{2}*{\bf Case }& \multirow{2}*{\bf LB-$\ell_0$} & \multirow{2}*{\bf \#-$\lambda$} & \multicolumn{2}{|l|}{\bf Convex-IP-$\ell_0$} & \multicolumn{2}{|l|}{\bf Pert-Convex-$\text{IP}_0$} & \multicolumn{2}{|l|}{\bf SDP} \\ \cline{4 - 9} 

& & & gap & Time & gap & Time & gap & Time \\ \hline

\bf Case 1 (200, 20) & 11222.152 & 2 & 0.779 \% & 7205 &  \bf 0.041 \% & 2391& -\% & -\\ 
\bf  &  &  & &  &   0.042 \% & 2178 &  &   \\ 
\bf  &  &  & &  &   0.466 \% & 3707&  &    \\
\hline

\bf Case 2 (200, 20) & 14588.507 & 2 & 0.503 \% & 7205 & \bf 0.032 \% & 1285 & -\% & -\\ 
\bf  &  &  & &  &   0.036 \% & 2772 &  &     \\ 
\bf  &  &  & &  &   0.479 \% & 7212 &  &  \\
\hline

\bf Case 1 (300, 20) & 8282.32 & 3 & 13.336 \% & 7212 &  \bf 0.089 \% & 2745 &- \% & - \\ 
\bf  &  &  &  &  &   0.159 \% & 1386 &  &     \\ 
\bf  &  &  &  &  &   1.523 \% & 7227&  &   \\
\hline

\bf Case 2 (300, 20) & 10233.583 & 3 & 4.182 \% & 7210 &   0.078 \% & 1835&-\% & -\\ 
\bf  &  &  &  & & \bf  0.07 \% & 99 &  &    \\ 
\bf  &  &  &  & &   0.817 \% & 7229&  &   \\
\hline

\bf Case 1 (400, 20) & 12976.349 & 3 & 55.172 \% & 7219 & \bf 0.08 \% & 2563& -\% &- \\ 
\bf  &  &  &  &  &  0.105 \% & 5278 &  &     \\ 
\bf  &  &  &  &  &  4.288 \% & 7248&  &  \\
\hline

\bf Case 2 (400, 20) & 11809.325 & 2 & 45.209 \% & 7219 & 0.082 \% & 4257 & -\% & -\\ 
\bf  &  &  &  &  &  0.084 \% & 6934&  &    \\ 
\bf  &  &  &  &  & \bf 0.08 \% & 485 &  &  \\
\hline

\bf Case 1 (500, 20) & 10218.591 & 3 & 65.637 \% & 7231 & \bf 0.13 \% & 3882& -\% & -\\ 
\bf  &  &  &  &  &  0.142 \% & 6568&  &    \\ 
\bf  &  &  &  &  &  2.067 \% & 7288&  &    \\
\hline

\bf Case 2 (500, 20) & 11032.377 & 3 & 48.034 \% & 7229 & \bf 0.114 \% & 6603 & -\% &-  \\ 
\bf  &  &  &  &  &  0.138 \% & 2753&  &    \\ 
\bf  &  &  &  &  &  4.88 \% & 7280&  &   \\
\hline

\bf Case 1 (1000, 20) & 10193.919 & 3 & - \% & - &  1.38 \% & 303 & -\% &- \\ 
\bf  &  &  &  &  &  1.358 \% & 1707  &  &   \\ 
\bf  &  &  &  &  & \bf 0.24 \% & 3257   &  &  \\
\hline

\bf Case 2 (1000, 20) & 13867.929 & 3 & - \% & - &  0.691 \% & 318 & -\% &- \\ 
\bf  &  &  &  &  &  0.674 \% & 1927 &  &    \\ 
\bf  &  &  &  &  & \bf 0.18 \% & 8807 &  &   \\
\hline

\end{tabular}}
\end{center}
\end{table}

\pagebreak

\begin{table}[ht]
\caption {Controlling Sparsity - Cardinality 10} \label{tab:controlling-sparsity}
\begin{center}
\resizebox{0.9 \textwidth}{!}{\begin{tabular}{|l|l|l|l|l|l|l|l|l|l|l|l|l|l|l|l|} \hline
	
\multirow{2}*{\bf Case }& \multirow{2}*{\bf LB-$\ell_0$} & \multirow{2}*{\bf \#-$\lambda$} & \multicolumn{2}{|l|}{\bf Convex-IP-$\ell_0$} & \multicolumn{2}{|l|}{\bf Pert-Convex-$\text{IP}_0$} & \multicolumn{2}{|l|}{\bf SDP} \\ \cline{4 - 9} 

& & & gap & Time & gap & Time &  gap & Time \\ \specialrule{.15em}{.05em}{.05em} 

\bf Case 1 (200, 10) & 706 & 1 & \bf 0.14 \% & 925 & 2.9 \% & 117 & 0.42 \% & 1360  \\ 
\bf  &  & &  & &   2.6 \% & 340 &  &     \\ 
\bf  &  & &  & &  2.6 \% & 3663 &  &    \\ 
\hline
\bf Case 2 (200, 10) & 680 & 1 & \bf 0.14 \% & 1195 & 3.53 \% & 176 & 1.2 \% & 1148 \\ 
\bf  &  & &  & &   3.38 \% & 372 &  &    \\ 
\bf  &  & &  & &   3.53 \% & 3672 &  &   \\ 
\hline
\bf Case 1 (300, 10) & 972 & 1 & \bf 1.4 \% & 1958 & 3.91 \% & 135 & NaN & - \\ 
\bf  &  & &  & &   3.81 \% & 453 &  &    \\ 
\bf  &  & &  & &   3.70 \% & 3635 &  &    \\ 
\hline
\bf Case 2 (300, 10) & 976 & 1 & \bf 1.1 \% & 3007 & 3.79 \% & 278 & NaN & - \\ 
\bf  &  & &  & &   3.48 \% & 1558 &  &     \\ 
\bf  &  & &  & &   3.69 \% & 3772 &  &     \\ 
\hline
\bf Case 1 (400, 10) & 1239 & 1 & \bf 1.3 \% & 7207 &  4.21 \% & 769 & NaN & - \\ 
\bf  &  & &  & &  3.96 \% & 699 &   &  \\ 
\bf  &  & &  & &  3.96 \% & 3699 &   & \\ 
\hline
\bf Case 2 (400, 10) & 1207 & 1 & \bf 1.6 \% & 7206 & 3.56 \% & 221 & NaN & - \\ 
\bf  &  & &  & &   3.48\% & 1894 &    &   \\ 
\bf  &  & &  & &   3.40 \% & 3697 &    &  \\ 
\hline
\bf Case 1 (500, 10) & 1498 & 1 & \bf 2.1 \% & 12180 & 5.21 \% & 1026 & NaN & - \\ 
\bf  &  & &  & &   4.74 \% & 2881 &   &  \\ 
\bf  &  & &  & &   4.81 \% & 3661 &   &  \\ 
\hline
\bf Case 2 (500, 10) & 1498 & 1 & \bf 2.1 \% & 13917 & 4.14 \% & 251 & NaN & - \\ 
\bf  &  & &  & &   4.07 \% & 1039 &  &    \\ 
\bf  &  & &  & &   4.01 \% & 3783 &  &    \\ 
\hline
\bf Case 1 (1000, 10) & 3948 & 1 &  - & - &  59.7 \% & 2206 & NaN & - \\ 
\bf  &  & &  &  &  53.3 \% & 8318 &    &  \\ 
\bf  &  & &  &  &  \bf 49.5 \% & 3600 &    &  \\ 
\hline
\bf Case 2 (1000, 10) & 4002 & 1 &  NaN & - & 58.1 \% & 3270 & NaN & - \\ 
\bf  &  & &  &  &  51.0 \% & 8356 &    &  \\ 
\bf  &  & &  &  &  \bf 47.6 \% & 3600 &    &  \\ 
\hline
\end{tabular}}
\end{center}
\end{table}

\pagebreak

\begin{table}[H]
\caption {Controlling Sparsity - Cardinality 20} \label{tab:controlling-sparsity-20}
\begin{center}
\resizebox{0.9 \textwidth}{!}{\begin{tabular}{|l|l|l|l|l|l|l|l|l|l|l|l|l|l|l|l|} \hline
	
\multirow{2}*{\bf Case }& \multirow{2}*{\bf LB-$\ell_0$} & \multirow{2}*{\bf \#-$\lambda$} & \multicolumn{2}{|l|}{\bf Convex-IP-$\ell_0$} & \multicolumn{2}{|l|}{\bf Pert-Convex-$\text{IP}_0$} & \multicolumn{2}{|l|}{\bf SDP} \\ \cline{4 - 9} 

& &  & gap & Time & gap & Time & gap & Time \\ \hline

\bf Case 1 (200, 20) & 1341.432 & 1 & 0.97 \% & 277 &  0.01 \% & 1434 & -\% &- \\ 
\bf  &  &  & &  &  \bf 0.009 \% & 4726&  &    \\ 
\bf  &  &  & &  &   0.735 \% & 2554&  &   \\
\hline

\bf Case 2 (200, 20) & 1287.45 & 1 & 1.63 \% & 332 &  0.009 \% & 887 & -\% &- \\ 
\bf  &  &  &  & & \bf 0.008 \% & 2847 &  &    \\ 
\bf  &  &  &  & &  1.22 \% & 1971 &  &     \\
\hline

\bf Case 1 (300, 20) & 1839.578 & 1 & 1.25 \% & 1019 & \bf 0.551 \% & 1932& -\% & -\\ 
\bf  &  &  &  & &   0.636 \% & 4854&  &     \\ 
\bf  &  &  &  & &   7.027 \% & 7280&  &    \\
\hline

\bf Case 2 (300, 20) & 1849.485 & 1 & 0.192 \% & 2217 & \bf 0.19 \% & 897 & -\% & -\\ 
\bf  &  &  &  & &   0.796 \% & 7229&  &    \\ 
\bf  &  &  &  & &   4.287 \% & 7226&  &    \\
\hline

\bf Case 1 (400, 20) & 2339.441 & 1 & \bf 1.45 \% & 907 &  2.140 \% & 4343& -\% & -\\ 
\bf  &  &  &  & &   5.47 \% & 7265&  &    \\ 
\bf  &  &  &  & &   9.847 \% & 7248&  &   \\
\hline

\bf Case 2 (400, 20) & 2273.785 & 1 & \bf 2.34 \% & 3106 &  3.572 \% & 3059& -\% &- \\ 
\bf  &  &  &  & &   5.864 \% & 5164 &  &     \\ 
\bf  &  &  &  & &  10.537 \% & 7249 &  &   \\
\hline

\bf Case 1 (500, 20) & 2870.013 & 1 & \bf 2.34 \% & 2773 &  3.376 \% & 6013& -\% & - \\ 
\bf  &  &  &  & &   4.077 \% & 10870&  &     \\ 
\bf  &  &  &  & &   5.572 \% & 7285&  &   \\
\hline

\bf Case 2 (500, 20) & 2832.149 & 1 & \bf 2.37 \% & 3015 &  3.539 \% & 5011& -\% & -\\ 
\bf  &  &  &  & &  5.087 \% & 7293&  &    \\ 
\bf  &  &  &  & &  5.063 \% & 7283&  &   \\
\hline

\bf Case 1 (1000, 20) & 7535.996 & 1 & -\% &  - &  31.656 \% & 7851&-\% & -\\ 
\bf  &  &  &  & &  27.151 \% & 721 &  &    \\ 
\bf  &  &  &  & &  \bf 25.326 \% & 7518&  &    \\
\hline

\bf Case 2 (1000, 20) & 7759.88 & 1 & - \% & - &  29.393 \% & 311 & -\% & -\\ 
\bf  &  &  &  & &   25.230 \% & 809&  &   \\ 
\bf  &  &  &  & &  \bf  23.433 \% & 7510&  &  \\
\hline

\end{tabular}}
\end{center}
\end{table}

\begin{table}[ht]
\caption {First six sparse principal components of Pitprops} \label{tab:pitprops}
\begin{center}
\resizebox{0.8 \textwidth}{!}{
\begin{tabular}{|l|l|l|l|l|l|l|l|l|l|l|l|} \hline
	\multirow{2}*{\bf Cardinality }& 
	\multirow{2}*{\bf LB-$\ell_0$} & \multicolumn{2}{|l|}{\bf Convex-IP-$\ell_0$} & \multicolumn{2}{|l|}{\bf Pert-Convex-IP} & 
	\multicolumn{2}{|l|}{\bf SDP} \\ % \cline{3 - 10}
	
	& & gap & Time & gap & Time & gap & Time  \\ \specialrule{.15em}{.05em}{.05em} 
	
	 \bf Cardinality 5 & 3.406 & 3.2 \% & 0.40 & 6.0 \% & 0.34 & \bf 1.5 \% & 3.70 \\ \hline 
	 \bf Cardinality 2 & 1.882 & 1.4 \% & 0.23 & 3.6 \% & 0.34 & \bf 0 \% & 2.49\\ \hline
	 \bf Cardinality 2 & 1.364 & 3.8 \% & 0.30 & 7.6 \% & 0.85 & \bf 1.0 \% & 2.69 \\ \hline
	 \bf Cardinality 1 & 1 & 1.8 \% & 0.75 & 3.5 \% & 1.02 & \bf 0 \% & 2.40  \\ \hline
	 \bf Cardinality 1 & 1 & 2.2 \% & 0.30 & 3.6 \% & 0.61 & \bf 0 \% & 2.42\\ \hline
	 \bf Cardinality 1 & 1 & 1.2 \% & 0.30 & 2.1 \% & 0.51 & \bf 0 \% & 2.32 \\ \hline
	 \bf Sum of above & 9.652 & 2.5 \% & 2.28 & 4.8 \% & 3.67 & \bf 0.7 \% & 16.02  \\ \hline
\end{tabular}}
\end{center}
\end{table}

\begin{table}[ht]
\caption {Biological and Internet Data - Cardinality 10} \label{tab:bio-internet-data}
\begin{center}
\resizebox{0.9 \textwidth}{!}{\begin{tabular}{|l|l|l|l|l|l|l|l|l|l|l|l|l|l|l|l|} \hline
	
\multirow{2}*{\bf Case }& \multirow{2}*{\bf LB-$\ell_0$} & \multirow{2}*{\bf \#-$\lambda$}  & \multicolumn{2}{|l|}{\bf Convex-IP-$\ell_0$} & \multicolumn{2}{|l|}{\bf Pert-Convex-$\text{IP}_0$} & \multicolumn{2}{|l|}{\bf SDP} \\ \cline{4 - 9} 

& & & gap & Time & gap & Time & gap & Time \\ \specialrule{.15em}{.05em}{.05em}

\bf Eisen-1 (79, 10) & 17.33 & 1 & 0.3 \% & 4.6 &  \bf 0.12 \% & 63 & 2.2 \% & 15 \\ 
\bf  &  &  &  & &   0.17 \% & 113 &  &  \\
\bf  &  &  &  & &  0.4 \% & 412 &  &   \\
\hline
\bf Eisen-2 (118, 10)& 11.71 & 1 & \bf 1.4 \% & 96 &  4.10 \% & 69 & 2.0 \% & 52 \\ 
\bf  &  &  &  & &    2.13 \% & 139 &  &   \\
\bf  &  &  &  & &   1.70 \% & 385 &  &   \\
\hline
\bf Colon (500, 10) & 2641 & 1 & 14.7 \% & 9000 & 27.7 \% & 708 & NaN & - \\ 
\bf  &  &  &  &  &  9.58 \% & 1181 &  &  \\
\bf  &  &  &  &  &  \bf 6.89 \% & 353 &  &    \\
\hline
\bf Lymphoma (500, 10) & 6008 & 3 & 20.7 \% & 3723 &  41 \% & 610 & NaN & -\\ 
\bf  &  &  &  & &   21 \% & 1526 &  &    \\
\bf  &  &  &  & &   \bf 17 \% & 2808 &  &    \\
\hline
\bf Reddit (2000, 10) & 1052 & 1 &  NaN & - &  3.59 \%  & 5663 & NaN & - \\
\bf  &  &  &  & &    \bf 2.142 \% & 8584 &  &    \\
\bf  &  &  &  & &    3.615 \% & 4318 &  &    \\
\hline
\end{tabular}}
\end{center}
\end{table}

\begin{table}[ht]
\caption {Biological and Internet Data - Cardinality 20} \label{tab:bio-internet-data-20}
\begin{center}
\resizebox{0.9  \textwidth}{!}{\begin{tabular}{|l|l|l|l|l|l|l|l|l|l|l|l|l|l|l|l|} \hline
	
\multirow{2}*{\bf Case }& \multirow{2}*{\bf LB-$\ell_0$} & \multirow{2}*{\bf \#-$\lambda$} & \multicolumn{2}{|l|}{\bf Convex-IP-$\ell_0$} & \multicolumn{2}{|l|}{\bf Pert-Convex-$\text{IP}_0$} & \multicolumn{2}{|l|}{\bf SDP}  \\ \cline{4 - 9} 

& & & gap & Time &gap & Time & gap & Time  \\ \specialrule{.15em}{.05em}{.05em} 

\bf Eisen-1 (79, 20) & 17.719 & 1 & 1.30 \% & 742 &  \bf 0.062 \% & 450& 2.37\% & 13 \\ 
\bf  &  &  &  & &   0.102 \% & 7928&  &   \\
\bf  &  &  &  & &   0.333 \% & 7205&  &   \\
\hline

\bf Eisen-2 (118, 20) & 19.323 & 1 & 2.02 \% & 64 &  1.309  \% & 283& 2.28\% & 53\\ 
\bf  &  &  &  &  & \bf 0.502 \% & 904&  &  \\
\bf  &  &  &  &  &  1.294 \% & 7206&  &   \\
\hline

\bf Colon (500, 20) & 4255.694 & 1 & 15.3 \% & 7230 &  16.537 \% & 4510& - \% & - \\ 
\bf  &  &  & &  &  \bf 5.77 \% & 2931&   & \\
\bf  &  &  & &  &   5.89 \% & 7286 &   & \\
\hline

\bf Lymphoma (500, 20) & 9082.158 & 2 & 18.7 \% & 7239 &  22.569 \% & 1677& - \% & - \\ 
\bf  &  &  &  & &   12.3 \% & 1442&  &   \\
\bf  &  &  &  & &  \bf 11.81 \% & 3721&  &   \\
\hline

\bf Reddit (2000, 20) & 1119.046 & 1 & - \% & - &  \bf 4.256 \% & 7920& - \% & -\\ 
\bf  &  &  &  & &   4.288 \% & 1677&  &   \\
\bf  &  &  &  & &   4.776 \% & 4274&  &   \\
\hline

\end{tabular}}
\end{center}
\end{table}

\section{Acknowledgements}
We would like to thank Munmun De Choudhury for providing us with the internet data set. We would like to thank the anonymous reviewers for their constructive comments that significantly improved the presentation of this paper. Rahul Mazumder acknowledges research support from ONR-N000141812298, NSF-IIS-1718258.

%#######################################################
%#######################################################
\bibliographystyle{plain}
\bibliography{Ref,rahul_dbm3} 

\begin{thebibliography}{10}

\bibitem{allen2011sparse}
Genevera~I Allen and Mirjana Maleti{\'c}-Savati{\'c}.
\newblock Sparse non-negative generalized {PCA} with applications to
  metabolomics.
\newblock {\em Bioinformatics}, 27(21):3029--3035, 2011.

\bibitem{Bagroy:2017:SMB:3025453.3025909}
Shrey Bagroy, Ponnurangam Kumaraguru, and Munmun De~Choudhury.
\newblock A social media based index of mental well-being in college campuses.
\newblock In {\em Proceedings of the 2017 CHI Conference on Human Factors in
  Computing Systems}, CHI '17, pages 1634--1646, New York, NY, USA, 2017. ACM.

\bibitem{bertsimas_berk_2016}
Lauren Berk and Dimitris Bertsimas.
\newblock Certifiably optimal sparse principal component analysis.
\newblock {\em technical report}, 2016.

\bibitem{berthet2013optimal}
Quentin Berthet and Philippe Rigollet.
\newblock Optimal detection of sparse principal components in high dimension.
\newblock {\em The Annals of Statistics}, 41(4):1780--1815, 2013.

\bibitem{bienstock1996computational}
Daniel Bienstock.
\newblock Computational study of a family of mixed-integer quadratic
  programming problems.
\newblock {\em Mathematical Programming}, 74(2):121--140, 1996.

\bibitem{bomze2013copositivity}
Immanuel~M Bomze and Gabriele Eichfelder.
\newblock Copositivity detection by difference-of-convex decomposition and
  $\omega$-subdivision.
\newblock {\em Mathematical Programming}, 138(1-2):365--400, 2013.

\bibitem{bonami2016solving}
Pierre Bonami, Oktay G{\"u}nl{\"u}k, and Jeff Linderoth.
\newblock Solving box-constrained nonconvex quadratic programs.
\newblock {\em Optimization online}, pages 26--76, 2016.

\bibitem{burer2005local}
Samuel Burer and Renato~DC Monteiro.
\newblock Local minima and convergence in low-rank semidefinite programming.
\newblock {\em Mathematical Programming}, 103(3):427--444, 2005.

\bibitem{burer2009old}
Samuel Burer and Anureet Saxena.
\newblock Old wine in a new bottle: The {MILP} road to {MIQCP}.
\newblock {\em Optimization Online}, 2009.

\bibitem{burer2009globally}
Samuel Burer and Dieter Vandenbussche.
\newblock Globally solving box-constrained nonconvex quadratic programs with
  semidefinite-based finite branch-and-bound.
\newblock {\em Computational Optimization and Applications}, 43(2):181--195,
  2009.

\bibitem{cadima1995loading}
Jorge Cadima and Ian~T Jolliffe.
\newblock Loading and correlations in the interpretation of principle
  compenents.
\newblock {\em Journal of Applied Statistics}, 22(2):203--214, 1995.

\bibitem{chan2015worst}
Siu~On Chan, Dimitris Papailiopoulos, and Aviad Rubinstein.
\newblock On the worst-case approximability of sparse {PCA}.
\newblock {\em arXiv preprint arXiv:1507.05950}, 2015.

\bibitem{jordan_07}
A.~d'Aspremont, L.~El. Ghaoui, M.~I. Jordan, and G.~R.~G. Lanckriet.
\newblock A direct formulation for sparse {PCA} using semidefinite programming.
\newblock {\em SIAM Review}, 49:434--448, 2007.

\bibitem{d2007full}
Alexandre d'Aspremont, Francis~R Bach, and Laurent~El Ghaoui.
\newblock Full regularization path for sparse principal component analysis.
\newblock In {\em Proceedings of the 24th international conference on Machine
  learning}, pages 177--184. ACM, 2007.

\bibitem{d2005direct}
Alexandre d'Aspremont, Laurent~E Ghaoui, Michael~I Jordan, and Gert~R
  Lanckriet.
\newblock A direct formulation for sparse {PCA} using semidefinite programming.
\newblock In {\em Advances in neural information processing systems}, pages
  41--48, 2005.

\bibitem{de2018using}
Marianna De~Santis, Franz Rendl, and Angelika Wiegele.
\newblock Using a factored dual in augmented lagrangian methods for
  semidefinite programming.
\newblock {\em Operations Research Letters}, 46(5):523--528, 2018.

\bibitem{d2014approximation}
Alexandre d’Aspremont, Francis Bach, and Laurent El~Ghaoui.
\newblock Approximation bounds for sparse principal component analysis.
\newblock {\em Mathematical Programming}, 148(1-2):89--110, 2014.

\bibitem{d2008optimal}
Alexandre d’Aspremont, Francis Bach, and Laurent~El Ghaoui.
\newblock Optimal solutions for sparse principal component analysis.
\newblock {\em Journal of Machine Learning Research}, 9(Jul):1269--1294, 2008.

\bibitem{frangioni2007sdp}
Antonio Frangioni and Claudio Gentile.
\newblock {SDP} diagonalizations and perspective cuts for a class of
  nonseparable miqp.
\newblock {\em Operations Research Letters}, 35(2):181--185, 2007.

\bibitem{hastie2015statistical}
Trevor Hastie, Robert Tibshirani, and Martin Wainwright.
\newblock {\em Statistical learning with sparsity}.
\newblock CRC press, 2015.

\bibitem{he2011algorithm}
Yunlong He, Renato~DC Monteiro, and Haesun Park.
\newblock An algorithm for sparse {PCA} based on a new sparsity control
  criterion.
\newblock In {\em Proceedings of the 2011 SIAM International Conference on Data
  Mining}, pages 771--782. SIAM, 2011.

\bibitem{jeffers1967two}
JNR Jeffers.
\newblock Two case studies in the application of principal component analysis.
\newblock {\em Applied Statistics}, pages 225--236, 1967.

\bibitem{jolliffe2003modified}
Ian~T Jolliffe, Nickolay~T Trendafilov, and Mudassir Uddin.
\newblock A modified principal component technique based on the lasso.
\newblock {\em Journal of computational and Graphical Statistics},
  12(3):531--547, 2003.

\bibitem{journee2010generalized}
Michel Journ{\'e}e, Yurii Nesterov, Peter Richt{\'a}rik, and Rodolphe
  Sepulchre.
\newblock Generalized power method for sparse principal component analysis.
\newblock {\em Journal of Machine Learning Research}, 11(Feb):517--553, 2010.

\bibitem{phdthesisKim}
Jinhak Kim.
\newblock {\em Cardinality Constrained Optimization Problems}.
\newblock PhD thesis, Purdue University, West Lafayette, Indiana, 8 2016.

\bibitem{Ma2013}
Shiqian Ma.
\newblock Alternating direction method of multipliers for sparse principal
  component analysis.
\newblock {\em Journal of the Operations Research Society of China},
  1(2):253--274, Jun 2013.

\bibitem{magdon2017np}
Malik Magdon-Ismail.
\newblock {NP}-hardness and inapproximability of sparse {PCA}.
\newblock {\em Information Processing Letters}, 126:35--38, 2017.

\bibitem{mazumder2017discrete}
Rahul Mazumder and Peter Radchenko.
\newblock The discrete dantzig selector: Estimating sparse linear models via
  mixed integer linear optimization.
\newblock {\em IEEE Transactions on Information Theory}, 63(5):3053--3075,
  2017.

\bibitem{nemhauser1988integer}
George~L Nemhauser and Laurence~A Wolsey.
\newblock {\em Integer and Combinatorial Optimization. Interscience Series in
  Discrete Mathematics and Optimization}.
\newblock 1988.

\bibitem{papailiopoulos2013sparse}
Dimitris Papailiopoulos, Alexandros Dimakis, and Stavros Korokythakis.
\newblock Sparse {PCA} through low-rank approximations.
\newblock In {\em International Conference on Machine Learning}, pages
  747--755, 2013.

\bibitem{pennebaker2001linguistic}
James~W Pennebaker, Martha~E Francis, and Roger~J Booth.
\newblock Linguistic inquiry and word count: {LIWC} 2001.
\newblock {\em Mahway: Lawrence Erlbaum Associates}, 71(2001):2001, 2001.

\bibitem{Saha:2017:MSS:3171581.3134727}
Koustuv Saha and Munmun De~Choudhury.
\newblock Modeling stress with social media around incidents of gun violence on
  college campuses.
\newblock {\em Proc. ACM Hum.-Comput. Interact.}, 1(CSCW):92:1--92:27, December
  2017.

\bibitem{tausczik2010psychological}
Yla~R Tausczik and James~W Pennebaker.
\newblock The psychological meaning of words: {LIWC} and computerized text
  analysis methods.
\newblock {\em Journal of language and social psychology}, 29(1):24--54, 2010.

\bibitem{vershynin2016high}
Roman Vershynin.
\newblock {\em High-Dimensional Probability An Introduction with Applications
  in Data Science}.
\newblock Draft, 2016.

\bibitem{wang2016statistical}
Tengyao Wang, Quentin Berthet, and Richard~J Samworth.
\newblock Statistical and computational trade-offs in estimation of sparse
  principal components.
\newblock {\em The Annals of Statistics}, 44(5):1896--1930, 2016.

\bibitem{wht_09}
DM. Witten, R.~Tibshirani, and T.~Hastie.
\newblock A penalized matrix decomposition, with applications to sparse
  principal components and canonical correlation analysis.
\newblock {\em Biostatistics}, 10(3):515--534, 2009.

\bibitem{yuan2011truncated}
Xiao-Tong Yuan and Tong Zhang.
\newblock Truncated power method for sparse eigenvalue problems.
\newblock {\em arXiv preprint arXiv:1112.2679}, 2011.

\bibitem{yuan2013truncated}
Xiao-Tong Yuan and Tong Zhang.
\newblock Truncated power method for sparse eigenvalue problems.
\newblock {\em Journal of Machine Learning Research}, 14(Apr):899--925, 2013.

\bibitem{zhang2012sparse}
Youwei Zhang, Alexandre d’Aspremont, and Laurent El~Ghaoui.
\newblock Sparse {PCA}: Convex relaxations, algorithms and applications.
\newblock In {\em Handbook on Semidefinite, Conic and Polynomial Optimization},
  pages 915--940. Springer, 2012.

\bibitem{zhang2002low}
Zhenyue Zhang, Hongyuan Zha, and Horst Simon.
\newblock Low-rank approximations with sparse factors {I}: Basic algorithms and
  error analysis.
\newblock {\em SIAM Journal on Matrix Analysis and Applications},
  23(3):706--727, 2002.

\bibitem{zou2006sparse}
Hui Zou, Trevor Hastie, and Robert Tibshirani.
\newblock Sparse principal component analysis.
\newblock {\em Journal of computational and graphical statistics},
  15(2):265--286, 2006.

\end{thebibliography}
%#######################################################
%#######################################################

%\appendix
\newpage
\begin{APPENDICES}
\section{Notation} \label{sec:notation}

\begin{table}[!h]
\caption{Notation} 
\label{tab:notation}
\begin{center}
\begin{tabular}{|c|c|} 
\toprule
Notation & Definition\\ 
\midrule
$Y$ & data matrix of size $Y \in \mathbb{R}^{m \times n}$ \\
$A$ & sample covariance matrix $A = \frac{1}{m} Y^{\top} Y$ \\
$\|\cdot\|_0, \|\cdot\|_1, \|\cdot\|_2$ & $\ell_0, \ell_1, \ell_2$ norm \\
$k$ & sparsity parameter of sparse PCA problem \\
$\lambda^k(A)$ & optimal value of $\max_{\|x\|_0 \leq k, \|x\|_2 \leq 1} x^{\top} A x$ \\
$\text{conv}(S)$ & convex hull of set $S$ \\ 
$[n]$ & short notation of index set $\{1, \ldots, n\}$ \\
$\text{diag}(v)$ & diagonal matrix generated from a given vector $v$ \\
$\text{tr}(A)$ & trace of a matrix $A$ \\
$\text{OPT}_{\ell_1}$ & optimal value of $\max_{\|x\|_1 \leq \sqrt{k}, \|x\|_2 \leq 1} x^{\top} A x$ \\
$\rho$ & multiplicative approximation ratio between sparse PCA and its $\ell_1$ relaxation \\
$\{\lambda_i, v_i\}_{i = 1}^n$ & eigenpair of covariance matrix $A$ \\
$\{g_i\}_{i = 1}^n$ & continuous variable $g_i := x^{\top} v_i$ \\
$\{\theta_i\}_{i = 1}^n$ & upper bound of $g_i$ defined as $\theta_i = \max \{x^{\top} v_i: \|x\|_2 \leq 1, \|x\|_0 \leq k\}$  \\
$\{\gamma_i^{j}\}_{j = -N}^N$ & splitting points of interval $[- \theta_i, \theta_i]$ for each $i$ \\ 
$\{\xi_i\}_{i = 1}^n$ & piecewise linear upper approximation of $g_i^2$ \\
$s$ & upper bound of $\sum_{i \in \{i: \lambda_i < \lambda\}} - (\lambda_i - \lambda)(x^{\top} v_i)^2$ \\ 
$2N + 1$ & number of splitting points for interval $[- \theta_i, \theta_i]$ for each $i \in \{i: \lambda_i > \lambda\}$ \\
$\bar{\lambda}$ & $\bar{\lambda} := \max\{\lambda_i: \lambda_i \leq \lambda\}$ \\
$\{\lambda_{i_j}\}_{j = 1}^p$ & $\lambda_{i_1} \geq \cdots \geq \lambda_{i_p} \geq 0$ distinct values of eigenvalues of $A$ \\ 
$\Delta \lambda$ & eigenvalue gap $\Delta \lambda = \min \{\lambda_{i_j} - \lambda_{i_{j + 1}}\}$ for $j = 1, \ldots, p -1$\\
$\bar{A}$ & perturbed covariance matrix of $A$ \\
$(\bar{x}, \bar{y}, \bar{g}, \bar{\xi}, \bar{\eta}, \bar{s})$ & optimal solution for convex-IP \\
$\text{OPT}_{\text{convex-IP}}$ & optimal value of convex integer programming model \\
$\text{OPT}_{\text{pert-convex-IP}}$ & optimal value of perturbed convex integer programming model \\
$b_{(v)}$ & parameter used for cutting planes defined in Section 2.3.3\\
$S_k$ & feasible region of sparse PCA with sparsity parameter $k$ \\
$T_k$ & $\ell_1$ relaxation of sparse PCA with sparsity parameter $k$ \\
$I_{\text{pos}}$ & the size of set $\{i: \lambda_i > \lambda\}$ \\
$I_{\text{pos}}^{\text{ini}}$ & initial input of $\{i: \lambda_i > \lambda\}$ \\ 
$\text{iter}$ & number of iterations used for perturbed convex IP method \\
\bottomrule
\end{tabular}
\end{center}
\end{table}

\section{SDP relaxation}\label{sec:appsdp}
The \ref{eq:SPCA} problem $\max_{\|x\|_2 = 1, \|x\|_0 \leq k} x^{\top} A x$  is equivalent to a nonconvex problem:
\begin{align*}
	\max ~ & \text{tr}(AX) \notag\\
	\text{s.t. } & \text{tr}(X) = 1, \|X\|_0 \leq k^2, X \succeq 0, \text{rank}(X) = 1. %\tag{SDP} \label{eq:SDP}
\end{align*}
Further relaxing this by replacing its rank and cardinality constraints with $\mathbf{1}^{\top} |X| \mathbf{1} \leq k$ gives the standard SDP relaxation:
\begin{align}
	\max ~ & \text{tr}(AX) \notag\\
	\text{s.t. } & \text{tr}(X) = 1, \mathbf{1}^{\top} |X| \mathbf{1} \leq k, X \succeq 0. \tag{SDP} \label{eq:SDP}
\end{align}

\section{  Proof of Proposition \ref{prop:SPCA-bound}}\label{sec:app2}
\begin{proof}
	\textbf{Proof of Proposition \ref{prop:SPCA-bound}:}  Let $x^{\ast} = (x_i^{\ast})_{i = 1}^n$ be an optimal solution of \ref{eq:SPCA}. Then set 
	\begin{align*}
		& \left\{
		\begin{array}{llll}
			g_i^{\ast} & \gets & (x^{\ast})^{\top} v_i, & i \in [n],\\
			\left( (\eta_i^{- N})^{\ast}, \ldots, (\eta_i^{N})^{\ast} \right) & \gets & \left( \eta_i^{- N}, \ldots, \eta_i^{N} \right) \in \text{SOS-2 and } \sum_{j = -N}^{N} \gamma_i^j (\eta_i^j)^{\ast} = g_i^{\ast}, & i \in \{i: \lambda_i > \lambda\},\\ 
			\xi_i^{\ast} & \gets & \sum_{j = - N}^N (\gamma_i^j)^2 \eta_i^j, & i \in \{i: \lambda_i > \lambda\}, \\
			y_i^{\ast} & \gets & |x_i^{\ast}|, & i \in [n], \\
			s_i^{\ast} & \gets & \sum_{i \in \{i: \lambda_i \leq \lambda \}} - (\lambda_i - \lambda) g_i^{\ast}. & 
 		\end{array}
		\right.
	\end{align*}
	Note that the above solution $(x^{\ast}, y^{\ast}, g^{\ast}, \xi^{\ast}, \eta^{\ast}, s^{\ast})$ is a feasible solution for \ref{eq:convex-IP}.  This is easy to verify for all the constraints except the constraint
$\sum_{i\in \{i:\lambda_i > \lambda\}} \xi_i + \sum_{i\in \{i:\lambda_i \leq \lambda\}}g_i^2 \leq 1 + \frac{1}{4N^2}\sum_{i\in \{i:\lambda_i > \lambda\}}\theta_i^2.$ Note that to verify this constraint, it is sufficient  to verify that $\xi_i \leq g_i^2 + \frac{1}{4N^2}\theta_i^2$ for $i\in \{i:\lambda_i > \lambda\}$. This is easily verified based on the size of the discretization and the structure of SOS-2 constraints. 

Moreover, the objective value of feasible solution $(x^{\ast}, y^{\ast}, g^{\ast}, \xi^{\ast}, \eta^{\ast}, s^{\ast})$ is 
	\begin{align*}
		\lambda + \sum_{i \in \{i: \lambda_i > \lambda\}} (\lambda_i - \lambda) \xi_i^{\ast} - s^{\ast}
		\geq & \lambda + \sum_{i \in \{i: \lambda_i > \lambda\}} (\lambda_i - \lambda) ( g_i^{\ast} )^2 - s^{\ast} \\
		= & \lambda + \sum_{i \in \{i: \lambda_i > \lambda\}} (\lambda_i - \lambda) ((x^{\ast})^{\top} v_i)^2 + \sum_{i \in \{i: \lambda_i \leq \lambda\}} (\lambda_i - \lambda) ((x^{\ast})^{\top} v_i)^2 \\
		= & \lambda + \sum_{i = 1}^n (\lambda_i - \lambda) ((x^{\ast})^{\top} v_i)^2. 
	\end{align*}
	Note that the optimal solution $x^{\ast}$ of \ref{eq:SPCA} has property $\| x^{\ast} \|_2 = 1$ and $\sum_{i = 1}^n v_i v_i^{\top} = I_n$. Then $\lambda + \sum_{i = 1}^n (\lambda_i - \lambda) ((x^{\ast})^{\top} v_i)^2 = (x^{\ast})^{\top} A x^{\ast} = \lambda^k(A)$. Therefore, $\text{OPT}_{\text{convex-IP}} \geq \lambda^k(A)$.
\end{proof}

\section{ Proof of Proposition \ref{prop:convex-IP-bound}}\label{sec:app3}
\begin{proof}
	\textbf{Proof of Proposition \ref{prop:convex-IP-bound}:} Let $(\bar{x}, \bar{y}, \bar{g}, \bar{\xi}, \bar{\eta}, \bar{s})$ be an optimal solution for \ref{eq:convex-IP}. Its optimal value then satisfies the following:
	\begin{align*}
		\text{OPT}_{\text{convex-IP}} & = \lambda + \sum_{i \in \{i: \lambda_i > \lambda\}} (\lambda_i - \lambda) \bar{\xi}_i - \bar{s} \\
		& = \lambda + \sum_{i \in \{i: \lambda_i > \lambda\}} (\lambda_i - \lambda) \left( \bar{\xi}_i - \bar{g}_i^2 + \bar{g}_i^2 \right) - \bar{s} \\
		& = \lambda + \sum_{i \in \{i: \lambda_i > \lambda\}} (\lambda_i - \lambda)\left( \bar{\xi}_i - \bar{g}_i^2 \right) + \sum_{i \in \{i: \lambda_i > \lambda\}} (\lambda_i - \lambda) \bar{g}_i^2 - \bar{s}.  
	\end{align*}
	Since variable $s$ satisfies $\sum_{i \in \{i: \lambda_i \leq \lambda\}} - (\lambda_i - \lambda) g_i^2 \leq s$, to maximize the objective function, $\bar{s}$ should be equivalent to $\sum_{i \in \{i: \lambda_i \leq \lambda\}} - (\lambda_i - \lambda) \bar{g}_i^2$, then the above formula can be represented as
	\begin{align}
		& \lambda + \sum_{i \in \{i: \lambda_i > \lambda\}} (\lambda_i - \lambda)\left( \bar{\xi}_i - \bar{g}_i^2 \right) + \sum_{i \in \{i: \lambda_i > \lambda\}} (\lambda_i - \lambda) \bar{g}_i^2 - \bar{s} \nonumber\\
		= & \lambda + \sum_{i \in \{i: \lambda_i > \lambda\}} (\lambda_i - \lambda)\left( \bar{\xi}_i - \bar{g}_i^2 \right) + \sum_{i \in \{i: \lambda_i > \lambda\}} (\lambda_i - \lambda) \bar{g}_i^2 + \sum_{i \in \{i: \lambda_i \leq \lambda\}} (\lambda - \lambda) \bar{g}_i^2 \nonumber\\
		= & \sum_{i \in \{i: \lambda_i > \lambda\}} (\lambda_i - \lambda)\left( \bar{\xi}_i - \bar{g}_i^2 \right) + \bigg( \lambda + \sum_{i = 1}^n (\lambda_i - \lambda) \bar{g}_i^2 \bigg). \label{eq:propconvexIPbet}
	\end{align} 
	By previous results, $ \lambda + \sum_{i = 1}^n (\lambda_i - \lambda) \bar{g}_i^2 = \bar{x}^{\top} A \bar{x}$. Note that due to the $\ell_2-$norm constraint $\|x\|_2 \leq 1$ and the $\ell_1-$norm constraint present in \ref{eq:convex-IP} problem, we have $\bar{x} \in T_k = \{ x \in \mathbb{R}^n : \|x\|_2 \leq 1, \|x\|_1 \leq \sqrt{k}\} \subseteq \rho \cdot \text{Conv}\left(S_k\right)$. Therefore $\bar{x}^{\top} A \bar{x}$ is upper bounded by the value $\rho^2 \cdot \lambda^k(A)$. 
	
	{\color{black} To upper bound the first term in (\ref{eq:propconvexIPbet}), since $g_i = \sum_{j = - N}^N \gamma_i^j \eta_i^j, ~ \xi_i = \sum_{j = - N}^N (\gamma_i^j)^2 \eta_i^j$ for $i \in \{i: \lambda_i > \lambda\}$ and the SOS-2 construction enforces that there are at most two \textit{active} continuous SOS-2 variables $\eta_i^j, \eta_i^{j + 1}$ such that $\eta_i^j + \eta_i^{j + 1} = 1$ with $\eta_i^j, \eta_i^{j + 1} \geq 0$ and the other SOS-2 variables are all zeros, then 
	\begin{align*}
		\xi_i - g_i^2 & ~ = \sum_{j = - N}^N (\gamma_i^j)^2 \eta_i^j - \left( \sum_{j = - N}^N \gamma_i^j \eta_i^j \right)^2 & \\
		& ~ = (\gamma_i^j)^2 \eta_i^j + (\gamma_i^{j + 1})^2 \eta_i^{j + 1} - \left( \gamma_i^j \eta_i^j + \gamma_i^{j + 1} \eta_i^{j + 1} \right)^2 & \text{ for $\eta_i^j, \eta_i^{j + 1}$ active} \\
		& ~ = (\gamma_i^{j + 1} - \gamma_i^j)^2 \eta_i^j (1 - \eta_i^j) & \text{ via $\eta_i^j + \eta_i^{j + 1} = 1$} \\
		& ~ \leq \max_{j = -N, \ldots, N - 1} (\gamma_i^{j + 1} - \gamma_i^j)^2 \cdot \frac{1}{4} 
	\end{align*}
	where in all possible partition of $[- \theta_i, \theta_i]$, the evenly partition of $[- \theta_i, \theta_i]$ achieves the minimum value of $\max_{j = -N, \ldots, N - 1} (\gamma_i^{j + 1} - \gamma_i^j)^2 = \frac{\theta_i^2}{N^2}$. Hence (\ref{eq:propconvexIPbet}) can be upper bounded as follows: }
	\begin{align*}
		\text{OPT}_{\text{convex-IP}} & =  \sum_{i \in \{i: \lambda_i > \lambda\}} (\lambda_i - \lambda)\left( \bar{\xi}_i - \bar{g}_i^2 \right) + \bigg( \lambda + \sum_{i = 1}^n (\lambda_i - \lambda) \bar{g}_i^2 \bigg) \\
		& \leq  \frac{1}{4 N^2} \sum_{i \in \{i: \lambda_i > \lambda\}} (\lambda_i - \lambda) \theta_i^2 + \rho^2 \cdot \lambda^k(A).
	\end{align*}
\end{proof}

{\color{black}
\section{Appendix: Proof of Proposition \ref{prop:convex-IP-poly}}
\begin{proof}
	\textbf{Proof of Proposition \ref{prop:convex-IP-poly}:} Given the heuristic lower bound $\lambda$, the number of splitting points $N$, the size of set $I_{\text{pos}} = |\{i: \lambda_i > \lambda\}|$, for each $i \in \{i: \lambda_i > \lambda\}$, there are at most $2N$ possible choices of \textit{active} SOS-2 variables, i.e., 
	\begin{align*}
		\eta_i^j, \eta_i^{j + 1} > 0, \text{ for } j = - N, \ldots, 0, \ldots, N - 1.  
	\end{align*}
	Thus there are at most $(2N)^{|I_{\text{pos}}|}$ choices of \textit{active} SOS-2 variables for a \ref{eq:convex-IP} problem. For a fixed value of \textit{active} SOS-2 variables, the \ref{eq:convex-IP} problem reduces to be a continuous convex optimization problem which can be solved exactly within polynomial time, say $T$. Thus the \ref{eq:convex-IP} can be solved within $(2N)^{|I_{\text{pos}}|} \cdot T$. 
\end{proof}
}

\section{Proof of Proposition \ref{prop:Pert-Convex-IP}}
\begin{proof}
	\textbf{Proof of Proposition \ref{prop:Pert-Convex-IP}:} Based on Proposition \ref{prop:convex-IP-bound}, we have
	\begin{align*}
		\text{OPT}_{\text{Pert-Convex-IP}} \leq \rho^2 \lambda^k(\bar{A}) + \frac{1}{4N^2} \sum_{i \in \{i: \lambda_i > \lambda\}} (\lambda_i - \lambda) \theta_i^2. 
	\end{align*}
Note that $\bar{A} - A = \sum_{i \in \{i: \lambda_i \leq \lambda\}} (\bar{\lambda} - \lambda_i) v_i v_i^{\top}$. Therefore, 
	\begin{align*}
		\rho^2 \lambda^k(\bar{A}) & = \rho^2 \lambda^k\left(A + (\bar{A} - A) \right)  \\
		& \leq ~ \rho^2 \lambda^k(A) + \rho^2 \lambda^k(\bar{A} - A) \\
		&  \leq ~ \rho^2 \lambda^k(A) + \rho^2 (\bar{\lambda} - \lambda_{\min} (A)) .
	\end{align*} 
\end{proof}

\section{Convex-IP Method and Pert-Convex-IP Method}\label{sec:appconIP}
Algorithm \ref{algo:convex-IP-method} presents all the details of the convex IP solved. Algorithm \ref{algo:Pert-Convex-IP-method} presents all the details of the Pert-Convex-IP solved. 
\begin{algorithm}[h!]
\caption{Convex-IP Method} \label{algo:convex-IP-method}
\begin{algorithmic}[1]
\State \emph{Input}: Sample covariance matrix $A$, cardinality constraint $k$, size of set $\{i: \lambda_i > \lambda \}$ we desire, number of one branch splitting points $N$.
\State \emph{Output}: Lower and upper bound of \ref{eq:SPCA} or \ref{eq:ell_1-relax} based on the choice of $\theta_i$. 
\Function{Convex-IP Method}{$A, k, I_{\text{pos}}, N$}\label{function:convex-IP-method}
\State Set lower bound and warm starting point $(\text{LB}, \bar{x}) \gets \Call{Heuristic Method}{A, k, x^0}$.
\State Set parameter $\lambda_{I_{\text{pos} + 1}} \leq \lambda \leq \text{LB}$ if possible, otherwise set $\lambda \gets \text{LB}$.
\State Set splitting points $\gamma_i^j$ as above based on $N$ and the choice of $\theta_i$, see Section \ref{sec:cip} [\ref{sec:cip-point3}] .
\State To warm start, add additional splitting points based on the point $\bar{x}$.
\State Add cutting-plane (\ref{eq:cut}) to the model based on the choice of $\theta_i$. 
\State Run \ref{eq:convex-IP} problem.
\State Set $\text{UB} \gets$ \ref{eq:convex-IP} if running to the optimal, or the current dual bound obtained from \ref{eq:convex-IP}.
\State \Return $\text{LB, UB}$. 
\EndFunction
\end{algorithmic}
\end{algorithm}

\begin{algorithm}[h!]
\caption{Pert-Convex-IP Method} \label{algo:Pert-Convex-IP-method}
\begin{algorithmic}[1]
\State \emph{Input}: Sample covariance matrix $A$, cardinality constraint $k$, size of set $\{i: \lambda_i > \lambda \}$ we desire, number of one branch splitting points $N$, maximum number of iterations $\text{iter}$.
\State \emph{Output}: Lower and upper bound of \ref{eq:SPCA} or \ref{eq:ell_1-relax} based on the choice of $\theta_i$. 
\Function{Pert-Convex-IP Method}{$A, k, I_{\text{pos}}, N, \text{iter}$}\label{function:Pert-Convex-IP-method}
\State Set lower bound and warm starting point $(\text{LB}, \bar{x}) \gets \Call{Heuristic Method}{A, k, x^0}$.
\State Set parameter $\lambda_{I_{\text{pos} + 1}} \leq \lambda \leq \text{LB}$ if possible, otherwise set $\lambda \gets \text{LB}$.
\State Set parameter $\bar{\lambda} \triangleq \max \{\lambda_i: \lambda_i \leq \lambda \} < \lambda$ if possible.  
\State Set splitting points $\gamma_i^j$ as above based on $N$ and the choice of $\theta_i$, see Section \ref{sec:cip} [\ref{sec:cip-point3}].
\State To warm start, add additional splitting points based on the point $\bar{x}$.
\While{current iteration does not exceed the maximum number of iterations $\text{iter}$ or time limit is not up}
\State Run \ref{eq:pert-convex-IP} problem.
\State Set $\text{UB} \gets$ \ref{eq:pert-convex-IP} if running to the optimal, or the current dual bound obtained from \ref{eq:pert-convex-IP}.
\State Set $\hat{x} \gets$ current feasible solution obtained from \ref{eq:pert-convex-IP}.
\State Add additional splitting points based on solution obtained in solving \ref{eq:pert-convex-IP} problem. 
\State Add cutting-plane (\ref{eq:cut}) to the model based on the choice of $\theta_i$. 
\EndWhile
\State \Return $\text{LB, UB}$.
\EndFunction
\end{algorithmic}
\end{algorithm}

\section{Description of Data Sets} \label{sec:data-set}
\subsection{Artificial Data Sets}
We first conduct numerical experiments on three types of artificial data sets, denoted as the spiked covariance recovery from the paper \cite{papailiopoulos2013sparse}, the synthetic example from the paper \cite{zou2006sparse}, and the controlling sparsity case from the paper \cite{d2005direct}.  A description of each of these three types of instances is presented below:
\subsubsection{Spiked covariance recovery}
Consider a covariance matrix $\Sigma$, which has two sparse eigenvectors with dominated eigenvalues and the rest eigenvector are unconstrained with small eigenvalues. Let the first two dominant eigenvectors $v_1, v_2$ of $\Sigma$ be: 
\begin{align}
	& [v_1]_i = \left\{
	\begin{array}{ll}
		\frac{1}{\sqrt{10}} & i = 1, \ldots, 10, \\
		0 & \text{ otherwise }
	\end{array}
	\right., & ~ 
	[v_2]_i = \left\{
	\begin{array}{ll}
		\frac{1}{\sqrt{10}} & i = 11, \ldots, 20, \\
		0 & \text{ otherwise }
	\end{array}
	\right., 
\end{align}
with the eigenvalues corresponding to the first two dominant eigenvectors be $\lambda_1 \gg 1$ and $\lambda_2 \gg 1$, and the remaining eigenvalues be 1. For example, in our numerical experiments, set $\Sigma \gets 399 \cdot v_1 v_1^{\top} + 299 \cdot v_2 v_2^{\top} + I$.

We have four distinct settings under the spiked covariance recovery case. Let $n$ be the number of features, i.e., the size of the sample covariance matrix of our numerical cases. Let $m$ be the number of samples we generated. We set $n = \{200, 300, 400, 500, 1000\}$ and $m = \{ 50 \}$. Therefore, under each setting of $n$, we generate $m$ random samples $x_i \sim N(0, \Sigma)$, and get our sample covariance matrix $\hat{\Sigma} = \frac{1}{50} \sum_{i = 1}^{50} x_i x_i^{\top}$. In Table \ref{tab:toy-example}, for each setting, we repeat the experiment for 2 times (case 1, case 2), and compare the dual bounds obtained from all three methods. 

\subsubsection{Synthetic Example}
Given $n$, let $n_1, n_2, n_3 \in \left\{ \lceil \frac{n}{3}\rceil, \lfloor \frac{n}{3} \rfloor \right\}$ such that $n_1 + n_2 + n_3 = n$. Let $\mathbf{0}_{p \times q}$ be the matrix of all zeros with size $p \times q$. Let $\mathbf{1}_p$ be the vector of all ones with length $p$. 

Then:
\begin{align}
	\Sigma = \begin{pmatrix}
		290 \cdot \mathbf{1}_{n_1} \mathbf{1}_{n_1}^{\top} + I_{n_1} & \mathbf{0}_{n_1 \times n_2} & - 87 \cdot \mathbf{1}_{n_1} \mathbf{1}_{n_3}^{\top} \\
		\mathbf{0}_{n_2 \times n_1} & 300 \cdot \mathbf{1}_{n_2} \mathbf{1}_{n_2}^{\top} + I_{n_2} & 277.5 \cdot \mathbf{1}_{n_2} \mathbf{1}_{n_3}^{\top} \\
		- 87 \cdot \mathbf{1}_{n_3} \mathbf{1}_{n_1}^{\top} & 277.5 \cdot \mathbf{1}_{n_3} \mathbf{1}_{n_2}^{\top} & 582.7875 \cdot \mathbf{1}_{n_3} \mathbf{1}_{n_3}^{\top} + I_{n_3} 
	\end{pmatrix}.
\end{align} 

In our experiments, we set $n = \{200, 300, 400, 500, 1000\}$, and generate $m = 50$ samples such that $x_i \sim N(0, \Sigma)$. Again, the sample empirical covariance matrix is $\hat{\Sigma} = \frac{1}{50} \sum_{i = 1}^{50} x_i x_i^{\top}$. In Table \ref{tab:artificial-data}, for each setting of $n$, we repeat the experiment twice (case 1, case 2), and compare dual bounds obtained from all three methods. 

\subsubsection{Controlling Sparsity}
Like the spiked covariance recovery case, the covariance matrix $\Sigma$ of controlling sparsity case can also be represented as the summation of a term generated by sparse eigenvector with dominated eigenvalue and the remaining part with small eigenvalues. 
%But in controlling sparsity case, 
Generate a $n \times n$ matrix $U$ with uniformly distributed coefficients in $[0,1]$ which can be seen as white noise. Let $v \in \{0,1\}^n$ be a sparse vector with $\|v\|_0 \leq k$. We then form a test matrix $\Sigma = U^{\top} U + \sigma v v^{\top}$, where $\sigma$ is the signal-to-noise ratio and is set to 15.

In our experiments, we set $n = \{200, 300, 400, 500, 1000\}$ and generate $m = 50$ samples $x_i \sim N(0, \Sigma)$ for $i = 1, \ldots, 50$. Therefore the sample empirical covariance matrix is $\hat{\Sigma} = \frac{1}{50} \sum_{i = 1}^{50} x_i x_i^{\top}$. In Table \ref{tab:controlling-sparsity}, for each setting of $n$, we repeat the experiment twice (case 1, case 2), and compare dual bounds obtained from all three methods.

\subsection{Real Data Sets}
We conduct numerical experiments on three types of real data sets, the benchmark pitprops data from \cite{jeffers1967two}, biological data from  \cite{d2007full,papailiopoulos2013sparse,yuan2013truncated} and large-scale data collected from internet. 

\subsubsection{Pitprops Data}
The PitProps data set in \cite{jeffers1967two} (consisting of 180 observations with 13 measured variables) has been a standard benchmark to evaluate algorithms for sparse PCA.

Based on previous work, we also consider the first six $k-$sparse principal components. Note the $i$-th $k-$sparse principal component $x^i$ is obtained by solving $\arg \max_{\|x\|_2 = 1, \|x\|_0 \leq k} x^{\top} A^i x$ where $A^1 \gets A$ and $A^i \gets (I - x^{i - 1} (x^{i - 1})^{\top}) A^{i - 1} (I - x^{i - 1} (x^{i - 1})^{\top})$ for $i = 2, \ldots, 6$. Table \ref{tab:pitprops} lists the six extracted sparse principal direction with cardinality setting $5-2-2-1-1-1$.

\subsubsection{Biological Data}
In Table \ref{tab:bio-internet-data} we present numerical experiments on four biological data sets. The first two biological data sets (Eisen-1, Eisen-2) are from \cite{yuan2013truncated}. The Colon cancer data set is from Alon et al. (1999). The Lymphoma data set is from Alizadeh et al. (2000). 

\subsubsection{Large-scale Internet Data}

In Table \ref{tab:bio-internet-data} we also present numerical experiments on internet dataset. This dataset is constructed out of textual posts shared on the popular social media Reddit. Based on prior work~\cite{Bagroy:2017:SMB:3025453.3025909, Saha:2017:MSS:3171581.3134727}, the archive of all public Reddit posts shared on Google’s Big Query was utilized to obtain a set of 3292 posts from the subreddit r/stress from December 2010 to January 2017. The r/stress community allows individuals to self-report and disclose their stressful experiences and is a support community. For example, two (paraphrased) post excerpts say: ``Feel like I am burning out (again...) Help: what do I do?"; and ``How do I calm down when I get triggered?". The community is also heavily moderated; hence these 3292 posts were considered to be indicative of actual stress.~\cite{Saha:2017:MSS:3171581.3134727}.

Then on this collected set of posts, standard text-based feature extraction techniques were applied per post, starting with cleaning the data (stopword elimination, removal of noisy words, stemming), and then building a language model with the n-grams in a post ($n$=2). The outcomes of this language model provided us with 1950 features, after including only the top most statistically significant features. Additionally, the psycholinguistic lexicon Linguistic Inquiry and Word Count (LIWC)~\cite{pennebaker2001linguistic} was leveraged  to obtain features aligning with 50 different empirically validated psychological categories, such as positive affect, negative affect, cognition, and function words. These features have been extensively validated in prior work to be indicative of stress and similar psychological constructs~\cite{tausczik2010psychological}. Our final dataset matrix comprised 3092 rows, corresponding to the 3092 posts, and 2000 features in all.

The purpose of testing the sparse PCA technique on this dataset is to identify those features that are theoretically guaranteed to be the most salient in describing the nature of stress expressed in a post. In turn, these salient features could be utilized by a variety of stakeholders like clinical psychologists, and community moderators and managers to gain insights into stress-related phenomenon as well as to direct interventions as appropriate.

The final $A$ matrix can be found on the website: \newline https://www2.isye.gatech.edu/~sdey30/publications.html

{\color{black}
\section{Comparison with Existing Primal Heuristics for Lower Bounds}
\label{sec:LB-comparison}
In this section, we compare our method Algorithm~\ref{algo:heuristic} for obtaining good primal feasible solutions with two standard heuristics methods for sparse PCA in the literature: truncated power method (TPM, \cite{yuan2011truncated}), generalized power method (GPM, \cite{journee2010generalized}) with $\ell_0$-penalty. See Table~\ref{tab:LB} for a comparison on all the real instances. 
\begin{table}[!h]
\caption{Compare with existing primal methods}
\label{tab:LB}
\vskip - 0.15in
\begin{center}
\begin{small}
\begin{sc}
\resizebox{1 \textwidth}{!}{
\begin{tabular}{|l||c|c||c|c||c|c|c|c|}
\toprule
%Instance & SPCA-Primal (Our method) & TPM & GPM \\ 

\multirow{2}*{Instance} & \multicolumn{2}{|c||}{SPCA-Primal (Our method)} & \multicolumn{2}{|c||}{TPM} & \multicolumn{2}{|c|}{GPM} \\ \cline{2 - 7}

 & LB & Time & LB & Time & LB & Time  \\ 

\midrule
Pitprops $k = 5$ & \bf 3.406 & 0.1  & \bf 3.406 & 0.0  & \bf 3.406 & 0.1  \\
Eisen-1 $k = 10$ & \bf 17.335 & 0.0  & \bf 17.335 & 0.0  & \bf 17.335 & 2.3  \\
Eisen-2 $k = 10$ & \bf 11.718 & 0.0  & \bf 11.718 & 0.0  & 11.605 & 4.1  \\
CovColon $k = 10$ & \bf 2641.228 & 0.4  & \bf 2641.228 & 0.4  & \bf 2641.228 & 59.7  \\
Lymp $k = 10$ & \bf 5911.412 & 0.3  & \bf 5911.412 & 0.2  & 5753.563 & 81.4 \\
Reddit $k = 10$ & \bf 1052.020 & 7.4  & \bf 1052.020 & 4.5  & \bf 1052.020 & 1881.4  \\

\bottomrule
\end{tabular}
}

\end{sc}
\end{small}
\end{center}
\vskip -0.1in
\end{table}
As we can see, all the methods produce solutions with more or less the same objective function values. 
}

{\color{black}
\section{Comparison with Existing Methods for Dual Bounds}
\label{sec:UB-comparison}
In this section, we compare the performance of our convex integer program method with (1) Mosek, in our experience one of the best commercial implementations of SDP solvers; and (2) two variants of the approach presented in \cite{de2018using}, which uses the main idea of \cite{burer2005local}. The variants are listed as follows:

\begin{enumerate}
	\item \textbf{DADAL:} Directly using code available online from \cite{de2018using}: Dual Alternating Direction Augmented Lagrangian (DADAL) method can be used to find out the upper bounds of the SDP problem. In order to use the freely available implementation, the DADAL method requires the remodeling of the original problem into the following standard format:
		\begin{align*}
			\min \langle \bm{A}, \bm{X} \rangle & ~ \text{s.t} ~ \mathcal{A}(\bm{X}) = \bm{b}, ~ \bm{X} \succeq \bm{0}.
		\end{align*}

		Thus to find the dual bounds of the sparse PCA with covariance matrix of size $d$, we need to (1) add additional auxiliary variables for inequality constraints, (2) reformulate the variables into a p.s.d. matrix. For the step-(1), the original sparse PCA problem can be formulated in the following fashion: 
\begin{align*}
	\min & ~ \langle - \bm{A}, \bm{X} \rangle & \tag{SDP-equality} \label{eq:SDP-format} \\
	\text{s.t.} & ~ \langle \bm{I}_d, \bm{X} \rangle + \mu_1 = 1 \\
	& ~ \langle \bm{I}_{d^2}, \text{diag}(\bm{Y}) \rangle + \mu_2 = k \\
	& ~ \langle \bm{E}_{ij}^+, \bm{X} \oplus \text{diag}(\bm{Y}) \rangle + \gamma_{ij}^+ = 0, ~ \forall ~ ij \\
	& ~ \langle \bm{E}_{ij}^-, \bm{X} \oplus \text{diag}(\bm{Y}) \rangle + \gamma_{ij}^- = 0 , ~ \forall ~ ij \\
	& ~ \bm{X}, \text{diag}(\bm{Y}), \text{diag}(\bm{\gamma}^+), \text{diag}(\bm{\gamma}^-), \text{diag}(\mu) \succeq \bm{0}	
\end{align*}
where $\oplus$ is the direct sum of two matrices, i.e., $\bm{A} \oplus \bm{B} := \begin{pmatrix} \bm{A} & 0 \\ 0 & \bm{B} \\ \end{pmatrix}$, the matrix $\text{diag}(\bm{Y})$ is a short notation of $\text{diag}(\text{vec}(\bm{Y}))$ with $\text{vec}(\bm{Y})$ the vectorization of matrix $\bm{Y}$, and the matrix $\bm{E}_{ij}^+, \bm{E}_{ij}^-$ are 
\begin{align*}
	& \bm{E}_{ij}^+ := \begin{pmatrix}
		E_{ij}  &  0 \\
		0 & - \text{diag}(\text{vec}(E_{ij})) \\
	\end{pmatrix}, ~ \bm{E}_{ij}^- := \begin{pmatrix}
		- E_{ij}  &  0 \\
		0 & - \text{diag}(\text{vec}(E_{ij})) \\
	\end{pmatrix}, & \forall i,j \in [d] \times [d] 
\end{align*}
with $E_{ij} \in \mathbb{R}^{d \times d}$ the standard basis matrix (i.e., the component $(i,j)$ equals to 1, and the rest components equal to 0). Rewrite the variables of \ref{eq:SDP-format} into a p.s.d. matrix 
\begin{align*} 	
	\tilde{\bm{X}} := \begin{pmatrix}
		\bm{X} & & & & \\
		& \text{diag}(\bm{Y}) & & & \\
		& & \text{diag}(\bm{\gamma}^+) & & \\
		& & & \text{diag}(\bm{\gamma}^-) & \\
		& & & & \mu \\
	\end{pmatrix} \in \mathbb{R}^{(d + 3 d^2 + 2) \times (d + 3 d^2 + 2)}.
\end{align*}
For the step-(2), the \ref{eq:SDP-format} can be further transferred into the standard SDP format as follows:
\begin{align*}
	\min & ~ \langle - \bm{A} \oplus \bm{0}_{d^2} \oplus \bm{0}_{d^2} \oplus \bm{0}_{d^2} \oplus \bm{0}_{2} , \tilde{\bm{X}} \rangle \tag{standard-SDP} \label{eq:standard-SDP}  \\
	\text{s.t.} & ~  \langle \bm{I}_d \oplus \bm{0}_{d^2} \oplus \bm{0}_{d^2} \oplus \bm{0}_{d^2} \oplus \text{diag}(1,0), \tilde{\bm{X}} \rangle = 1 \\
	& ~ \langle \bm{0}_d \oplus \bm{I}_{d^2} \oplus \bm{0}_{d^2} \oplus \bm{0}_{d^2} \oplus \text{diag}(0,1), \tilde{\bm{X}} \rangle = k \\
	& ~ \langle (\bm{E}_{ij}^+ + \bm{E}_{ij}^+) \oplus (\text{diag}(\text{vec}(E_{ij})) + \text{diag}(\text{vec}(E_{ji}))) \oplus \bm{0}_{d^2} \oplus \bm{0}_2, \tilde{\bm{X}} \rangle = 0, ~ \forall i \geq j \\
	& ~ \langle (\bm{E}_{ij}^- + \bm{E}_{ij}^-) \oplus \bm{0}_{d^2} \oplus (\text{diag}(\text{vec}(E_{ij})) + \text{diag}(\text{vec}(E_{ji}))) \oplus \bm{0}_2, \tilde{\bm{X}} \rangle = 0, ~ \forall i \geq j \\
	& ~ \tilde{\bm{X}} \succeq \bm{0} 
\end{align*}
with the size of variable matrix $n = d + 3d^2 + 2$ and the number of linear constraints $m = 2 + d \times (d + 1)$. The code of DADAL method is downloaded from the author's \cite{de2018using} homepage  \footnote{\url{https://www.math.aau.at/or/Software/}}.
	\item \textbf{DADAL-SPCA:} A DADAL-SPCA method designed by us (which uses the main ideas of the DADAL method) works specifically for the sparse PCA problem. As we have seen above, using the standard code of DADAL involves increasing dimension to $(d + 3 d^2 + 2)^2$ which appears to be quiet inefficient for solving the standard SDP relaxation of sparse PCA. Therefore we alternatively pursued the following approach: Consider the primal and dual SDP relaxation of sparse PCA, 
		\begin{align*}
			& \begin{array}{rlll}
				\text{Primal} := \min_{\bm{X}, \bm{Y}} &
				~ \langle - \bm{A}, \bm{X} \rangle & \\
				\text{s.t.} & ~ \langle \bm{I}, \bm{X} \rangle \leq 1 & ~ (\mu_1 \geq 0) \\
				& ~ \langle \bm{1} \bm{1}^{\top}, \bm{Y} \rangle \leq k & ~ (\mu_2 \geq 0) \\
				& ~ \bm{Y} \geq \bm{X} & ~ (\bm{\gamma}^+ \geq 0) \\
				& ~ \bm{Y} \geq - \bm{X} & ~ (\bm{\gamma}^- \geq 0) \\
				& ~ \bm{X} \succeq \bm{0} & ~ (\bm{Z} \succeq \bm{0})
			\end{array} ~ & 
			\begin{array}{rllll}
			\text{Dual} := \max & ~ - \mu_1 - \mu_2 k \\
			\text{s.t.} & ~ \mu_1 \bm{I} + \bm{\gamma}^+ - \bm{\gamma}^- - \bm{A} - \bm{Z} = \bm{0} \\
			& ~ \mu_2 \bm{1} \bm{1}^{\top} - \bm{\gamma}^+ - \bm{\gamma}^- = \bm{0} \\
			& ~ \bm{Z} \succeq \bm{0} \\
			& ~ \mu_1, \mu_2, \bm{\gamma}^+, \bm{\gamma}^- \geq 0
			\end{array}
		\end{align*} 
		with its augmented Lagrangian 
		\begin{align*}
			\mathcal{L}_{\sigma}(\bm{\mu}, \bm{\gamma}, \bm{Z}; \bm{X}, \bm{Y}) := & ~ - \mu_1 - \mu_2 k + \langle \bm{M}_1, \bm{X} \rangle + \langle \bm{M}_2, \bm{Y} \rangle - \frac{\sigma}{2} \|\bm{M}_1\|_F^2 -  \frac{\sigma}{2} \|\bm{M}_2\|_F^2,
		\end{align*}
		where $\bm{M}_1, \bm{M}_2$ are defined as 
		\begin{align*}
			\bm{M}_1 := &~  \mu_1 \bm{I} + \bm{\gamma}^+ - \bm{\gamma}^- - \bm{A} - \bm{Z}, \\
			\bm{M}_2 := & ~ \mu_2 \bm{1} \bm{1}^{\top} -  \bm{\gamma}^+ -  \bm{\gamma}^-. 
		\end{align*}
		We initialize $\bm{X}^0, \bm{Y}^0, \bm{Z}^0$ as follows: Compute eigenvalue decomposition of $\bm{A} = \bm{V} \bm{\Lambda}_{\bm{A}} \bm{V}^{\top}$, let $\bm{v}_1$ be the leading eigenvector of $\bm{V}$ with respect to the largest eigenvalue. Set 
		\begin{align*}
		    \bm{X}^0 \gets & ~ \bm{v}_1 \bm{v}_1^{\top}, \\
		    \bm{Y}^0 \gets & ~ |\bm{X}^0|, \\
		    \bm{Z}^0 \gets & ~ \bm{0},
 		\end{align*}
 		along with the starting augmented Lagrangian parameter $\sigma^0$. 
		In $(k + 1)$-th iteration, update each variable based on the following rule which is similar as the DADAL method proposed in \cite{de2018using}.
		\begin{align*}
			\bm{\mu}^{k + 1}, \bm{\gamma}^{k + 1} \gets & ~ \argmax_{\bm{\mu} \geq 0, \bm{\gamma} \geq 0} \mathcal{L}_{\sigma^k} (\bm{\mu}, \bm{\gamma}, \bm{Z}^k; \bm{X}^k, \bm{Y}^k) \\
			\bm{Z}^{k + 1} \gets & ~ \left( - \frac{\bm{X}^k}{\sigma^k} + \bm{\mu}_1^{k + 1} \bm{I} + (\bm{\gamma}^+)^{k + 1} - (\bm{\gamma}^-)^{k + 1} - \bm{A} \right)_{\succeq 0} \\
			\bm{X}^{k + 1} \gets & ~ - \sigma \cdot \left( - \frac{\bm{X}^k}{\sigma^k} + \bm{\mu}_1^{k + 1} \bm{I} + (\bm{\gamma}^+)^{k + 1} - (\bm{\gamma}^-)^{k + 1} - \bm{A} \right)_{\preceq 0} \\
			\bm{Y}^{k + 1} \gets & ~ |\bm{X}^{k + 1}| \\
			\text{Update } \sigma & ~ \text{based on Algorithm 1 in \cite{de2018using}}
		\end{align*}
		where $(\bm{A})_{\succeq 0}, (\bm{A})_{\preceq 0}$ denote the positive semi-definite, negative semi-definite part of symmetric matrix $\bm{A}$. That is: Let $\bm{A} = \bm{U} \bm{\Sigma} \bm{U}^{\top}$ be its eigenvalue decomposition. Represent $\bm{\Sigma} = \bm{\Sigma}^+ + \bm{\Sigma}^-$ where $\bm{\Sigma}^+_{ii} = \max\{\bm{\Sigma}_{ii}, 0\}$ and $\bm{\Sigma}^-_{ii} = \min\{\bm{\Sigma}_{ii}, 0\}$, then 
		\begin{align*}
			(\bm{A})_{\succeq 0} := & ~ \bm{U} \bm{\Sigma}^+ \bm{U}^{\top}, \\
			(\bm{A})_{\preceq 0} := & ~ \bm{U} \bm{\Sigma}^- \bm{U}^{\top}. 
		\end{align*} 
		
	\begin{remark}
	The way we update our dual variables (and primal variables) in each iteration, there is no guarantee that the dual variables satisfy the equality constraints in the dual, namely,
    \begin{align*}
		\bm{M}_1 := &~  \mu_1 \bm{I} + \bm{\gamma}^+ - \bm{\gamma}^- - \bm{A} - \bm{Z} = 0, \\
		\bm{M}_2 := & ~ \mu_2 \bm{1} \bm{1}^{\top} -  \bm{\gamma}^+ -  \bm{\gamma}^- = 0. 
	\end{align*}
    Therefore, it is not true that we can always obtain exact dual bounds from every iteration. We store the dual bounds of iterations where the equality constraints are satisfied within a tolerance of $0.01$, i.e., 
    \begin{align*}
        \|\bm{M}_1\|_F + \|\bm{M}_2\|_F \leq 0.01.
    \end{align*}
    Moreover, after the final iteration, we add one more step by solving the following \emph{linear program}, 
    \begin{align*}
    \begin{array}{rll}
        \mu^{\mathrm{final}}, \bm{\gamma}^{\mathrm{final}} := \argmax_{\mu, \bm{\gamma}} & ~ - \mu_1 - \mu_2 k \\
			\text{s.t.} & ~ \mu_1 \bm{I} + \bm{\gamma}^+ - \bm{\gamma}^- - \bm{A} - \bm{Z}^{\mathrm{final}} = \bm{0}, \\
			& ~ \mu_2 \bm{1} \bm{1}^{\top} - \bm{\gamma}^+ - \bm{\gamma}^- = \bm{0}, \\
			& ~ \mu_1, \mu_2, \bm{\gamma}^+, \bm{\gamma}^- \geq 0,
	\end{array} \tag{final-dual} \label{eq:final-dual}
    \end{align*}
    where $\bm{Z}^{\mathrm{final}} \succeq 0$ is the dual variable obtained in the final step of DADAL-SPCA. It is easy to observe that $(\mu^{\mathrm{final}}, \bm{\gamma}^{\mathrm{final}}, \bm{Z}^{\mathrm{final}})$ is a dual feasible solution, and therefore a dual bound can be obtained from this dual feasible solution. 
    \end{remark}

	\emph{Stopping criteria:} The stopping criteria includes three conditions. Meeting any of the criteria stops the DADAL-SPCA algorithm. 
    \begin{enumerate}
        \item The maximum number of iteration is set to be $200$.
        \item The stopping criteria quantity $\delta$ proposed in  Algorithm 1 \cite{de2018using} is set to be 0.001, i.e., at the end of each iteration, we compute the primal and dual infeasibility errors as follows:
        \begin{align*}
            r_P := & ~ \frac{\max\{\text{Tr}(X) - 1, 0\} + \max\{ \langle \bm{1} \bm{1}^{\top}, \bm{Y} \rangle - k, 0\} }{1 + \sqrt{1 + k^2} }, \\
            r_D := & \frac{\|\bm{M}_1\|_F + \|\bm{M}_2\|_F}{1 + \|\bm{A}\|_F}, 
        \end{align*}
        and set $\delta := \max\{r_P, r_D\}$. 
        \item Since there is no closed form solution of the following updating step:
    \begin{align*}
        \bm{\mu}^{k + 1}, \bm{\gamma}^{k + 1} \gets & ~ \argmax_{\bm{\mu} \geq 0, \bm{\gamma} \geq 0} \mathcal{L}_{\sigma^k} (\bm{\mu}, \bm{\gamma}, \bm{Z}^k; \bm{X}^k, \bm{Y}^k),
    \end{align*}
    we use commercial solver Gurobi (called via Python) to solve this quadratic programming sub-problem in each iteration. For small instances (i.e., $d < 500$, Pitprops, Eisen-1, Eisen-2), the total time limit given for Gurobi solver is $3600$ seconds (1 hour); and for middle-size instance (i.e., $d = 500$, CovColon, Lymp), the total time limit given for Gurobi solver is $7200$ seconds (2 hours), and for large instance (i.e., $d = 2000$, Reddit), the total time limit given for Gurobi solver is $18000$ seconds (5 hours).
    \end{enumerate}
\end{enumerate}

Algorithm~\ref{algo:dual-bound-DADAL-SPCA} is the pseudocode of finding dual bounds using DADAL-SPCA.

\begin{algorithm}[h!]
\caption{Dual Bound DADAL-SPCA} \label{algo:dual-bound-DADAL-SPCA}
\begin{algorithmic}[1]
\State \emph{Input}: Covariance matrix $\bm{A}$, sparsity parameter $k$, maximum number of iteration $ T_{\max}$, total time limit for solver $T_{\text{total}}$, starting Lagrangian augmented parameter $\sigma^0$. 
\State \emph{Output}: Dual bound of sparse PCA. 
\Function{Dual Bound Method}{$\bm{A}, k, T_{\max}, T_{\text{total}}$} \label{function:dual-bound-DADAL-SPCA}
\State Compute eigenvalue decomposition on $\bm{A}$, let $\bm{v}_1$ be its leading eigenvector. 
\State Initialize $\bm{X} \gets \bm{v}_1 \bm{v}_1^{\top}, \bm{Y} \gets |\bm{X}|, \bm{Z} \gets \bm{0}^{d \times d}, (\mu_1, \mu_2) \gets (0,0), \bm{\gamma}^{\pm} \gets \bm{0}^{d \times d}$.
\State Run DADAL-SPCA with stopping criteria described above with starting Lagrangian augmented parameter $\sigma^0 \in \{0.001, 0.01, 0.1, 1\}$, and return $\text{UB}^{\text{DADAL-SPCA}}$.
\State Solve \ref{eq:final-dual} for a dual bound $\text{UB}^{\text{final-dual}}$.
\State \Return $\text{UB} \gets \min\{\text{UB}^{\text{final-dual}}, ~  \text{UB}^{\text{DADAL-SPCA}}\}$.
\EndFunction
\end{algorithmic}
\end{algorithm}

The gap obtained by DADAL-SPCA as described above with various values of $\sigma^0$ is reported in Table~\ref{tab:DADAL-SPCA-comparison}.

\begin{table}[!h]
\caption{DADAL-SPCA under different starting augmented Lagrangian parameter $\sigma^0$.}
\label{tab:DADAL-SPCA-comparison}
\vskip - 0.15in
\begin{center}
\begin{small}
\begin{sc}
\resizebox{ \textwidth}{!}{
\begin{tabular}{|l||c||c|c||c|c||c|c||c|c||c|c|c|}
\toprule
%Instance & Model-in-Paper & DADAL \cite{de2018using} & DADAL-SPCA & Mosek \\ 

\multirow{2}*{Instance $\backslash ~ \sigma^0$} &
\multirow{2}*{LB} & 
\multicolumn{2}{|c||}{$\sigma^0 = 0.001$} &
\multicolumn{2}{|c||}{$\sigma^0 = 0.01$} & 
\multicolumn{2}{|c||}{$\sigma^0 = 0.1$} &
\multicolumn{2}{|c||}{$\sigma^0 = 1$} 
\\ \cline{3 - 10}

 & & gap $\%$ & Time & gap $\%$ & Time & gap $\%$ & Time & gap $\%$ & Time \\

\midrule

Pitprops $k = 5$ & 3.406 & 3.96 & 6 & 1.79 & 5 & 1.70 & 2 & \bf 1.64 & 3 \\

Eisen-1 $k = 10$ & 17.33 & 2.23 & 270 & \bf 2.19  & 225 & 11.07 & 294 & 39.10 & 288 \\

Eisen-2 $k = 10$ & 11.71 & 2.32 & 1053 & 2.37 & 610  & \bf 2.08 & 898  & 2.12 & 897  \\

CovColon $k = 10$ & 2641 & 14.16 & 7492 & \bf 13.51 & 7281 & 19.05  & 7369  & 26.82  & 7301  \\

Lymp $k = 10$ & 6008 & \bf 29.67 & 7339  & 34.79  & 7331  & 46.84  & 7367  & 59.09  & 7373  \\

Reddit $k = 10$ & 1052 & - & O.M. & - & O.M. & - & O.M. & - & O.M. \\
\bottomrule
\end{tabular}
}
\end{sc}
\end{small}
\end{center}
\vskip -0.1in
\end{table}
 The ``Time'' column in Table~\ref{tab:DADAL-SPCA-comparison} denotes the total running time used for the DADAL-SPCA method. We can see that the ``Time'' of CovColon, Lymp reported in Table~\ref{tab:DADAL-SPCA-comparison} are greater than time limit for solver, 
    % ($3600$ seconds (Pitprops, Eisen-1, Eisen-2) or $7200$ seconds (CovColon, Lymp, Reddit)) 
    since additional time are required to implement the other four updating steps in each iteration. The out of memory (O.M.) for Reddit instance is due to the memory limitation to load Reddit instance $d = 2000$ for the update step  
    \begin{align*}
        \bm{\mu}^{k + 1}, \bm{\gamma}^{k + 1} \gets & ~ \argmax_{\bm{\mu} \geq 0, \bm{\gamma} \geq 0} \mathcal{L}_{\sigma^k} (\bm{\mu}, \bm{\gamma}, \bm{Z}^k; \bm{X}^k, \bm{Y}^k).
    \end{align*}
    We tried to solve the  \ref{eq:final-dual} linear program for Reddit instance, but the LP did not solve in 5 hours. (This LP has order $d^2$ variables, whereas the number of variables of convex integer program is order $d I_{\text{pos}} N$ and $I_{\text{pos}} N \ll d$ in this instance.)

To complete the comparison, we also list the comparison between our model in paper and DADAL, DADAL-SPCA, Mosek in Table~\ref{tab:UB-comparison}. 
    % The Table~\ref{tab:UB-comparison} had been proposed in Revision Round 2. 

\begin{table}[!h]
\caption{Compare with existing SDP methods}
\label{tab:UB-comparison}
\vskip - 0.15in
\begin{center}
\begin{small}
\begin{sc}
\resizebox{1 \textwidth}{!}{
\begin{tabular}{|l||c||c|c||c|c||c|c||c|c|c|}
\toprule
%Instance & Model-in-Paper & DADAL \cite{de2018using} & DADAL-SPCA & Mosek \\ 

\multirow{2}*{Instance} &
\multirow{2}*{LB} & 
\multicolumn{2}{|c||}{Model-in-Paper} & \multicolumn{2}{|c||}{DADAL \cite{de2018using}} & \multicolumn{2}{|c||}{DADAL-SPCA (best)} & \multicolumn{2}{|c|}{Mosek} \\ \cline{3 - 10}

 & & gap $\%$ & Time & gap $\%$ & Time & gap $\%$& Time & gap $\%$& Time \\

\midrule

Pitprops $k = 5$ & 3.406 & 3.26 & 0.4  & 82.43 & 593  & 1.64 & 3  & \bf 1.52 & 5 \\
Eisen-1 $k = 10$ & 17.33 & \bf 0.115 & 63  & - & O.M. & 2.19 & 225 & 2.19 & 15  \\
Eisen-2 $k = 10$ & 11.71 & \bf 1.71 & 385  & - & O.M. & 2.08 & 898 & 1.96 & 52 \\
CovColon $k = 10$ & 2641 & \bf 2.37 & 28  & - & O.M. & 13.51 & 7281  & - & O.M. \\
Lymp $k = 10$ & 6008 & \bf 17.86 & 4225  & - & O.M. & 29.67 & 7339 & - & O.M. \\
Reddit $k = 10$ & 1052 & \bf 2.24 & 8584  & - & O.M. & - & O.M. & - & O.M. \\

\bottomrule
\end{tabular}
}
\end{sc}
\end{small}
\end{center}
\vskip -0.1in
\end{table}

Based on Table~\ref{tab:UB-comparison}, we observe that the SDP-relaxation solved by Mosek produces the best bounds for the small instances (Pitprops, Eisen-1, Eisen-2), while DADAL-SPCA is able to produce bounds for Pitprops, Eisen-1, Eisen-2, CovColon, and Lymp. However, as we can see, except for Pitprops, the best dual bounds are obtained by solving convex IP model of this paper. 

}

\end{APPENDICES}

%%%%%%%%%%%%%%%%%
\end{document}